\documentclass[12pt,a4paper]{amsart}
\usepackage{amsfonts,amsmath,amssymb}
\usepackage{hyperref}
\usepackage{latexsym,fullpage,amsfonts,amssymb,amsmath,amscd,graphics,epic}
\usepackage[all]{xy}
\usepackage{amssymb,amsthm,amsxtra}
\usepackage[usenames]{color}
\usepackage{amscd}
\usepackage{amsthm}
\usepackage{amsfonts}
\usepackage{amssymb}
\usepackage{mathrsfs}
\usepackage{url}
\usepackage{bbm}
\usepackage{wasysym}
\usepackage{enumitem}

\theoremstyle{plain}
\newtheorem*{theorem*}{Theorem}
\newtheorem*{remark*}{Remark}
\newtheorem*{example*}{Example}
\newtheorem{lemma}{Lemma}[subsection]
\newtheorem{proposition}[lemma]{Proposition}
\newtheorem{corollary}[lemma]{Corollary}
\newtheorem{theorem}[lemma]{Theorem}

\newtheorem*{conjecture*}{Conjecture}

\newtheorem{introtheorem}{Theorem}
\newtheorem{introcorollary}[introtheorem]{Corollary}

\theoremstyle{definition}
\newtheorem{definition}[lemma]{Definition}

\theoremstyle{remark}
\newtheorem{remark}[lemma]{Remark}

\newcommand{\Hom}{\operatorname{Hom}}
\newcommand{\triv}{{\mathbbm 1}}
\newcommand{\id}{\operatorname{Id}}
\renewcommand{\Im}{\operatorname{Im}}
\newcommand{\Ind}{\operatorname{Ind}}
\newcommand{\Ker}{\operatorname{Ker}}
\newcommand{\Ext}{\operatorname{Ext}}

\newcommand{\Aut}{{\operatorname{Aut}}}
\newcommand{\End}{\operatorname{End}}

\newcommand{\Ob}{\operatorname{Ob}}

\newcommand{\Set}{\mathtt{Set}}
\newcommand{\Com}{\mathtt{Com}}
\newcommand{\Grps}{\mathtt{Grps}}
\newcommand{\cA}{\mathcal{A}}
\newcommand{\cC}{\mathcal{C}}
\newcommand{\cD}{\mathcal{D}}
\newcommand{\cT}{\mathcal{T}}
\newcommand{\cV}{\mathcal{V}}
\newcommand{\ex}{\mathit{ex}} 
\newcommand{\faith}{\mathit{faith}} 
\newcommand{\fg}{\mathfrak{g}}
\newcommand{\fX}{\mathfrak{X}}
\newcommand{\fY}{\mathfrak{Y}}
\newcommand{\fZ}{\mathfrak{Z}}
\newcommand{\Fun}{\mathrm{Fun}}
\newcommand{\Map}{\operatorname{Map}}
\newcommand{\Mod}{\operatorname{Mod}}
\newcommand{\sVect}{\mathtt{sVect}}

\newcommand{\InnaA}[1]{{#1}}
\newcommand{\InnaB}[1]{{#1}}
\newcommand{\InnaC}[1]{{#1}}
\newcommand{\InnaD}[1]{{#1}}

\def\quotient#1#2{%
    \raise1ex\hbox{$#1$}\Big/\lower1ex\hbox{$#2$}%
}

\begin{document}
 \renewcommand{\v}{\varepsilon} \newcommand{\p}{\rho}
\newcommand{\m}{\mu}

\def\Pic{{\bf Pic}}
\def\re{{\bf re}}
\def\e{{\bf e}}
\def\a{\alpha}
\def\ve{\varepsilon}
\def\b{\beta}
\def\D{\Delta}
\def\d{\delta}
\def\ad{\operatorname{ad}}
\def\f{{\varphi}}
\def\ga{{\gamma}}
\def\L{\Lambda}
\def\lo{{\bf l}}
\def\s{{\bf s}}
\def\A{{\bf A}}
\def\B{{\bf B}}
\def\cB{{\mathcal {B}}}
\def\C{{\mathbb C}}
\def\F{{\bf F}}
\def\G{{\mathfrak {G}}}
\def\g{{\mathfrak {g}}}
\def\gl{{\mathfrak {gl}}}
\def\tg{{\tilde{\mathfrak {g}}}}
\def\b{{\mathfrak {b}}}
\def\q{{\mathfrak {q}}}
\def\f{{\mathfrak {f}}}
\def\k{{\mathfrak {k}}}
\def\l{{\mathfrak {l}}}
\def\m{{\mathfrak {m}}}
\def\n{{\mathfrak {n}}}
\def\o{{\mathfrak {o}}}
\def\p{{\mathfrak {p}}}
\def\s{{\mathfrak {s}}}
\def\t{{\mathfrak {t}}}
\def\r{{\mathfrak {r}}}
\def\z{{\mathfrak {z}}}
\def\h{{\mathfrak {h}}}
\def\H{{\mathcal {H}}}
\def\O{\Omega}
\def\M{{\mathcal {M}}}
\def\T{{\mathbb {T}}}
\def\N{{\mathbb {N}}}
\def\U{{\mathcal {U}}}
\def\Z{{\mathbb Z}}
\def\P{{\mathcal {P}}}
\def\GVM{ GVM }
\def\iff{ if and only if  }
\def\add{{\rm add}}
\def\ld{\ldots}
\def\vd{\vdots}
\def\mod{{\rm mod}}
\def\len{{\rm len}}
\def\cd{\cdot}
\def\dd{\ddots}
\def\q{\quad}
\def\qq{\qquad}
\def\ol{\overline}
\def\tl{\tilde}
\def\xxRep{\text{$\underline{\operatorname{Rep}}$}}
\def\xxrep{\text{$\overline{\operatorname{Rep}}$}}
\def\Rep{\operatorname{\xxRep}}
\def\rep{\operatorname{\xxrep}}
\def\rp{\operatorname{Rep}}
\def\tV{\tilde{V}}
\def\tT{\tilde{T}}
\def\tX{\tilde{X}}
\def\tY{\tilde{Y}}

 \title{Deligne categories and the limit of categories $\rp(GL(m|n))$}
\author{ Inna Entova-Aizenbud, Vladimir Hinich and Vera Serganova }

\address{ Dept. of Mathematics, Hebrew University of Jerusalem, Givat Ram,
Jerusalem 91904, Israel.} 
\email{inna.entova@mail.huji.ac.il}
\address{ Dept. of Mathematics, University of Haifa, Mount Carmel, Haifa 31905, Israel.} 
\email{hinich@math.haifa.ac.il}
\address{ Dept. of Mathematics, University of California at Berkeley,
Berkeley, CA 94720.} 
\email{serganov@math.berkeley.edu}

\begin{abstract}
For each integer $t$ a tensor category $\cV_t$ is constructed, such that exact tensor functors $\cV_t\rightarrow \cC$ classify dualizable $t$-dimensional objects
in $\cC$ not annihilated by any Schur functor. This means that $\cV_t$ is the ``abelian envelope'' of the Deligne category $\cD_t=\rp(GL_t)$. Any tensor functor $\rp(GL_t)\longrightarrow\cC$ is proved to factor either through $\cV_t$ or through
one of the classical categories $\rp(GL(m|n))$ with $m-n=t$. The universal property
of $\cV_t$ implies that it is equivalent to the categories
 $\rp_{\cD_{t_1}\otimes\cD_{t_2}}(GL(X),\epsilon)$, ($t=t_1+t_2$, $t_1$ not integer) suggested by Deligne as candidates for the role of abelian envelope. 
\end{abstract}

\maketitle
\setcounter{tocdepth}{3}

\section{Introduction}
%
%

%
%
\subsection{}
In this paper we construct, for each integer $t$, a tensor category $\mathcal{V}_t$ satisfying a remarkable universal property: given a tensor category $\mathcal{C}$, the exact tensor functors $\mathcal{V}_t \rightarrow \mathcal{C}$ classify the $t$-dimensional objects in $\mathcal{C}$ not annihilated by any Schur functor. 


In \cite{DM}, Deligne and Milne constructed a family of rigid symmetric monoidal categories $\rp(GL_t)$ (denoted $\cD_t$ in this paper), parameterized by $t \in \mathbb{C}$; for $t$ non-integer, these are semisimple tensor categories satisfying the mentioned above universal property\footnote{The condition on Schur functors is void in this case.}.

For $t\in\Z$ the category $\cD_t$ is Karoubian but not abelian. The category $\mathcal{V}_t$ is abelian; it is built of pieces of categories of representations of Lie supergroups $GL(m|n)$ with $m-n = t$, and $\cD_t$ admits an embedding into $\mathcal{V}_t$ as a full rigid symmetric monoidal subcategory.

\subsection{}

%
Representation theory has evolved from the study of representations of groups by matrices to the study of the categories of representations of groups and, more generally supergroups. The Tannaka reconstruction theory allows one to recover the original (super) group from the category of its finite-dimensional representations, and Deligne's results (\cite{D1, D2}) give criteria on a tensor category to be the category of representations of a (super) group.

Not all tensor categories are categories of representations of groups or supergroups; obviously, a category with objects of non-integer dimensions are not such. In the category of finite dimensional representations of a group, the objects will obviously have non-negative integral dimensions, while for the supergroups, this dimension might be negative, but will still be an integer. 


In \cite{DM}, \cite{D}, Deligne and Milne constructed families of rigid symmetric monoidal categories $\rp(GL_t), \rp(Osp_t), \rp(S_t)$ ($t \in \mathbb{C}$) whose objects have not necessarily integral dimension. These categories possess some nice universal proerties in the $2$-category of symmetric monoidal categories. Additionally, their explicit description reflects many features of the classical representation theory.

%

The above categories are all Karoubian; they are semisimple at generic values of the parameter $t$ (in particular, when $t \notin \mathbb{Z}$).

When $t\in \mathbb{Z}$ ($t \in \mathbb{Z}_{\geq 0}$ for the family $\rp(S_t)$), the categories in question are not abelian; for $t=n \in \mathbb{Z}_{\geq 0}$, they admit a symmetric monoidal (SM for short) functor to the classical categories $\rp(GL_n), \rp(O_n), \rp(S_n)$, so that the classical categories appear as quotients of the respective Deligne categories.

On the other hand, there exist faithful SM functors from the 
Deligne categories at special values of $t$ to tensor (abelian) categories.  
Deligne conjectured that for each family and each value of $t$, there exists a
universal tensor category admitting an embedding of the Deligne category, and
suggested a construction of such a tensor category. In the case of the family
$\rp(S_t)$, this conjecture was proved by Comes and Ostrik in \cite{CO}. They also
gave an alternative construction of the tensor category using $t$-structures, which
leads to a rather explicit description of the abelian envelope of $\rp(S_t)$ (see \cite{En}).

\subsection{}
From now on, we will concentrate on the family $\rp(GL_t)$ which we denote $\cD_t$ for the sake of simplicity. 
The starting point of the present paper is the study of singular support for representations of Lie superalgebras, due to Duflo and Serganova, \cite{DS}. 

Let $\g$ be a Lie superalgebra, and let $x\in\g$ be an odd element satisfying $[x,x]=0$. One can define
a functor $\mathcal H$ from the category of $\g$-modules to the category of $\g_x$-modules, where $\g_x:=\Ker\ad_x/\Im \ad_x$. This functor sends a $\g$-module $M$ to
$M_x:=\Ker x/\Im x$. It is easy to see that this functor is symmetric monoidal.

We will apply this construction to $\g=\mathfrak{gl}(m|n)$ and $x$ an 
$(m+n)\times (m+n)$ matrix with $1$ in the upper-right corner and zero elsewhere. Then
$\g_x$ is isomorphic to $\mathfrak{gl}(m-1|n-1)$. 

This yields a collection of SM functors $$ \mathcal H: \rp(\mathfrak{gl}(m|n)) \longrightarrow \rp(\mathfrak{gl}(m-1|n-1))$$
The functors $\mathcal H$ are not exact, but are exact on certain subcategories of $\rp(\mathfrak{gl}(m|n))$. This allows us to construct a new  tensor category 
$\cV_t$, $t:=m-n$, together with a collection of SM functors $F_{m, n}: \cV_t \longrightarrow \rp(\mathfrak{gl}(m|n))$ which are compatible with the 
functors $ \mathcal H$. Note that the SM functors $F_{m, n}$ are not exact.

The category $\cV_t$ should be seen as an inverse limit of the system $(\rp(\mathfrak{gl}(m|n)), \mathcal H)$.

\subsection{}
The category $\cV_t$ is the main object of our study. We study the properties of this category using a variety of tools from the representation theory of the Lie supergroup $GL(m|n)$ and the Lie algebra $\gl(\infty)$ (see \cite{DPS}). We give a description of the isomorphism classes of simple objects in $\cV_t$, classify the blocks of this category, and show that $\cV_t$ is a union of highest weight categories.


%
The category $\cV_t$ admits a distinguished object $V_t$ of dimension $t$ coming from the standard representation $\mathbb{C}^{m|n}$ of $\mathfrak{gl}(m|n)$ ($m-n = t$). One has $F_{m,n}(V_t)=\C^{m|n}$ for all $m,n$.

Since the category $\cD_t$ is freely generated as a Karoubian symmetric monoidal category by an object of dimension $t$, there is a canonical SM functor $I: \cD_t  \longrightarrow \cV_t$. We prove that this functor is fully faithful (see Proposition \ref{prop:deligneconnection}). 


The tensor category $\cV_t$ enjoys the following remarkable universal property (see Theorems \ref{Mthm:uni-prop}, \ref{Mthm:uni-2}). 
\begin{introtheorem}\label{introthrm:universal}
 Let $\mathcal{C}$ be \InnaC{a tensor} category with an object $C$ of integral dimension $t$ which is not annihilated by any Schur functor. There exists an essentially unique exact SM functor $\cV_t \longrightarrow \mathcal{C}$ carrying $V_t$ to $C$. 
\end{introtheorem}

The proof of this theorem is based on two properties of the category $\cV_t$: 
\begin{itemize}
 \item Any object in $\cV_t$ can be presented as an image of an arrow $I(f)$, $f \in Mor(\cD_t)$.
 \item For any epimorphism $X\to Y$ in $\cV_t$ there exists a nonzero object $T\in\cD_t$
such that the epimorphism $X\otimes I(T)\to Y\otimes I(T)$ splits.
\end{itemize}

\subsection{}\label{introsec:abelian_env}
Another important result proved in this paper is connected to Deligne's philosophy 
of ``abelian envelope'': for special (that is, integral) values of $t$, one should not
expect to find a universal SM functor from $\cD_t$ to a tensor category $\mathcal{A}$
so that any SM functor $\cD_t \rightarrow \mathcal{B}$ to a tensor category
$\mathcal{B}$ factors through an exact SM functor 
$\mathcal{A} \longrightarrow \mathcal{B}$. Instead, one has a collection of such
functors $\cD_t$ to $\mathcal{A}_i$ so that any SM functor 
$\cD_t \rightarrow \mathcal{B}$ factors through one of $\mathcal{A}_i$. Among the
functors $\cD_t \longrightarrow \mathcal{A}_i$ only one is faithful, and it is called
 the abelian envelope of $\cD_t$.

If one considers the Deligne category $\rp(S_t)$ ($t = n \in \N$) instead of $\cD_t$, then it was shown in \cite{CO} that only two categories $\mathcal{A}_i$ appear: the classical category $\rp(S_n)$ and the abelian envelope ${\rp}^{ab}(S_{t=n})$. Here is our second result (see Theorem \ref{thrm:Deligne_conj}).

\begin{introtheorem}\label{introthrm:Deligne_conj} 
Let $\mathcal{T}$ be a \InnaC{tensor} category, and let $X$ be an object in $\mathcal{T}$
 of integral dimension $t$. Consider the canonical SM functor 
$$F_X:
\cD_t\longrightarrow \mathcal{T}
$$ carrying the $t$-dimensional generator of $\cD_t$ to $X$.
\begin{enumerate}[label=(\alph*)]
 \item If $X$ is not annihilated by any Schur functor then $F_X$ uniquely 
factors through the embedding $I:\cD_t\to\cV_t$ and gives rise to an exact SM functor $$\cV_t \longrightarrow \mathcal{T}$$
 sending $V_t$ to $X$.
 \item If $X$ is annihilated by some Schur functor then there exists a unique pair $m, n \in \mathbb{Z}_+$, $m-n=t$, such that $F_X$ factors through the SM functor $\cD_t\longrightarrow \rp(\gl(m|n))$ and gives rise to an exact SM functor 
 $$ \rp(\gl(m|n)) \longrightarrow \mathcal{T}$$
 sending the standard representation $\mathbb{C}^{m |n}$ to $X$.
\end{enumerate}

\end{introtheorem}
This answers a question posed by Deligne in \cite[Section 10]{D}.
The proof of this theorem is based on Theorem \ref{introthrm:universal} (in the form of Theorem \ref{Mthm:uni-prop}) and on the Tannakian formalism as described by Deligne in \cite{D1}.  

We remind once again that when $t  \notin \mathbb{Z}$, the Deligne category $\cD_t$ is semisimple, and thus a tensor category. This clearly implies that for $t  \notin \mathbb{Z}$, $\cD_t$ satisfies a property similar to Theorem \ref{introthrm:universal} (the condition on Schur functors is void in this case). This was detailed by Ostrik in \cite[Appendix B]{D}.
 
%
%

As a corollary, we identify $\cV_t$ with the following category constructed by Deligne.
Let $t_1\in\C-\Z$ and let $t_2=t-t_1$.
The category $\cD_{t_1}\otimes\cD_{t_2}$ has a canonical object $X$ of dimension $t$;
One has
\begin{introcorollary}(see Corollary~\ref{cor:Deligne_conj1})
  For any $t_1 \in \C-\Z$, there is a unique canonical equivalence
 $$\cV_t\longrightarrow\rp_{\cD_{t_1}\otimes\cD_{t_2}}(GL(X), \epsilon)$$
 carrying $V_t$ to $X$.
\end{introcorollary}

\subsection{}

Our results can be illustrated as follows.

The description of functors from Deligne categories to tensor categories looks like a category-theory version of the description of homomorphisms from a 
commutative ring to fields. Homomorphisms between fields are injective, simlarly to exact SM functors between tensor categories. This indicates that tensor 
categories should be considered as analogues of fields.

If $A$ is a commutative ring, any ring homomorphism 
$A\longrightarrow K$ to a field uniquely factors through a homomorphism
$A\longrightarrow k(\mathfrak{p})$ where $\mathfrak{p}$ is a prime ideal of $A$ 
and $k(\mathfrak{p})$ is the residue field of the localization $A_{\mathfrak{p}}$.
If $A$ is a domain, there is a fraction field $k(\mathfrak{p})$ with $\mathfrak{p}=0$, the only one for which the map $A\to k(\mathfrak{p})$ is injective. This is analogous to abelian envelope in our context (see \ref{introsec:abelian_env}). From this point of view, $\rp(S_t)$ looks like a local domain of dimension one.

On the other hand, $\cD_t=\rp(GL_t)$ is more curious. The existence of SM functors
$ \mathcal H: \rp(\mathfrak{gl}(m|n)) \longrightarrow \rp(\mathfrak{gl}(m-1|n-1))$
carrying the standard representation to the standard representation, implies that
the kernels of SM functors $\cD_t\to\rp(\mathfrak{gl}(m|n))$ form an infinite 
descending chain. The functor $\cD_t\to \rp(GL(m|n))$ is surjective when $m=0$ or $n=0$ --- this is the case where the category of representations of the 
supergroup (superalgebra) is semisimple. This means that the ``prime spectre'' of $\cD_t$ 
consists of the zero ideal (corresponding to $\cV_t$) and an infinite decreasing sequence of   the kernels of the functors $\cD_t\to\rp(GL(t+i|i))$ or 
for all $i\in\N$ such that $t+i\in\N$.

\subsection{Structure of the paper}

Sections \ref{sec:Deligne_prelim}, \ref{sec:rep_superalgebras}, \ref{sec:infty_rep_prelim}, we give a short overview of the required results from the theory of representations of superalgebras and the theory of Deligne categories.
Section \ref{sec:specialization_functor} sets the scene for the construction of the category $\cV_t$. Section \ref{sec:specialization_functor} describes the specialization functor from the category of representations of $\gl(\infty|\infty)$ to $\rp(\gl(m|n))$. Section \ref{ssec:V_lambda_objects} studies the standard objects in the subcategories 
$\rp^k(\gl(m|n))$ of the categories $\rp(\gl(m|n))$ of which $\cV_t$ will be ``glued''. This section contains some technical results which are important for the proofs of Theorem \ref{thrm:functorequiv} and Proposition \ref{prop:presentation}.

In Section \ref{sec:homology_functor} we construct the Duflo-Serganova homology functors $\mathcal H:\rp(\gl(m|n)) \to \rp(\gl(m-1|n-1)) $. We study their behaviour on the subcategories $\rp^k(\gl(m|n))$, showing that for $m, n>>k$, they induce equivalences $\rp^k(\gl(m|n))\to\rp^k(\gl(m-1|n-1))$.

In Section \ref{sec:maincategory} we construct the category $\cV_t$. 
We then study the properties of $\cV_t$: the functor from the Deligne category, blocks, translation functors, the highest weight structure of the ind-completion, and some auxilary results which are crucial in the proof of Theorem \ref{introthrm:universal} (such as Proposition \ref{prop:presentation}).

In Section \ref{sec:universal_prop} we formulate the universal property of $\cV_t$ (Theorem \ref{introthrm:universal}) in its most general form. 
The proof itself is contained in Section \ref{sec:proof_univ_prop}.

In Section \ref{sec:Deligne_conj} we recall Deligne's theory of Tannakian 
formalism, formulate and prove Theorem \ref{introthrm:Deligne_conj} in its most 
general form.

\section{Notation}

The base field throughout this paper will be $\mathbb{C}$.

\subsection{Tensor categories}
In this paper a {\it tensor category} is a rigid symmetric monoidal abelian $\mathbb{C}$-linear category, where the bifunctor $\otimes$ is bilinear on 
morphisms, and $\End(\triv) \cong \mathbb{C}$. An explicit definition can be found in \cite{D1, EGNO}. 
Note that in such a category the bifunctor $\otimes$ is biexact.

A functor between symmetric monoidal (SM) categories will be called a SM functor if respects the SM structure in the sense of \cite[2.7]{D1} 
(there it is called "foncteur ACU"); similarly, an $\otimes$-natural transformation between SM functors is a transformation respecting the 
monoidal structure in the sense of \cite[2.7]{D1}.

Notice that any such natural transformation is an isomorphism.

As in \cite[2.12]{D1}, a pre-Tannakian category is a tensor category satisfying finiteness conditions: namely, every object has finite length and every $\Hom$-space is finite-dimensional over $\mathbb{C}$.

\subsection{Partitions}
Given an integer $n \geq 0$, a weakly-decreasing sequence 

$\mu =(\mu_1, \mu_2, \ldots, \mu_k)$ of non-negative integers such that $\sum 
\mu_i = n$ is called a partition of $n$. The integer $n$ will be called the size 
of the partition, and $k$ will be called the length of the partition. We will 
often identify partitions which differ only by a tail of zeros.

Given a partition $\mu$, we denote by $\mu +\square$ the set of all partitions $\nu$ such that the vector $\mu-\nu$ (considered as an infinite vector) is a vector in the standard basis of $\mathbb Z^\infty$: this corresponds to the fact that the Young diagram of $\nu$ is obtained from the Young diagram of $\mu$ by adding one square.

A bipartition is an ordered pair of partitions: $\lambda = (\lambda^{\circ}, \lambda^{\bullet})$. 

For a bipartition $\lambda= (\lambda^{\circ}, \lambda^{\bullet})$, we will denote by $\lambda +\square$ (respectively, $\lambda+\blacksquare$) the set of bipartitions $\lambda'= (\lambda'^{\circ}, \lambda'^{\bullet})$ such that $\lambda'^{\circ} \in \lambda^{\circ} +\square$, $\lambda'^{\bullet} = \lambda^{\bullet}$ (resp., $\lambda'^{\circ} = \lambda^{\circ} $, $\lambda'^{\bullet}\in \lambda^{\bullet}+\square$).

\section{Preliminaries on Deligne categories}\label{sec:Deligne_prelim}
Let $\mathcal{D}^\circ_t$ be the free rigid symmetric monoidal $\mathbb{C}$-linear category generated by one object of dimension $t$ (denoted by Deligne as $\rp_0(GL_t)$, see \cite[Section 10]{D}). We denote by $X_t$ and $X^*_t$ the $t$-dimensional generator and its dual; we will also denote $X_t^{p, q} := X_t^{\otimes p} \otimes {X_t^*}^{\otimes q}$ the mixed tensor powers of $X_t$. 

These describe all the objects in $\mathcal{D}^\circ_t$ up to isomorphism. The morphisms are generated by morphisms $\id_{X_t}, ev_{X_t}$ and symmetric braiding under the operations $\circ, \otimes, *$, \cite{BCNR}. An explicit description of the spaces of morphisms in $\mathcal{D}^\circ_t$ in terms of walled Brauer algebras can be found in \cite{D, CW}. 

We will denote by $\mathcal{D}_t$ the Karoubi (additive) envelope of $\mathcal{D}^\circ_t$, which is obtained from $\mathcal{D}^\circ_t$ by adding formal direct sums and images of idempotents. The category $\mathcal{D}_t$ is a Karoubian rigid symmetric monoidal category, also called the Deligne category $\rp(GL_t)$; it is the universal Karoubian additive symmetric monoidal category generated by a dualizable object of dimension $t$. Its structure is studied in \cite{D, CW}.

We list below a few properties of $\mathcal{D}_t$:
\begin{itemize}
 \item For $m, n \in \Z_{\geq 0}$, the category $\mathcal{D}_{t = m-n}$ admits a full symmetric monoidal additive functor $F_{m,n}: \mathcal{D}_{t=m-n} \longrightarrow \rp(\gl(m|n))$.
 \item For $t \not\in \Z$, the category $\mathcal{D}_t$ is a semisimple abelian tensor category.
 \item For any $t$, the indecomposable objects (up to isomorphism) in the category $\mathcal{D}_t$ are parametrized by bipartitions (equivalently, pairs of Young diagrams of any size). 
 
 Let $t = m \in \Z_{\geq 0}$. The functor $F_{m,0}: \mathcal{D}_{t=m} \longrightarrow \rp(GL_m(\C))$ \InnaD{is full and} sends an indecomposable object $X_{\lambda^{\bullet}, \lambda^{\circ}}$ to the irreducible $GL_m(\C)$-representation with highest weight $\sum_i \lambda^{\bullet}_i \varepsilon_i - \sum_j \lambda^{\circ}_j \varepsilon_{m-j+1}$ whenever $\ell(\lambda^{\bullet}) + \ell(\lambda^{\circ}) \leq m$ (here $\ell$ is the number of rows in the Young diagram). Otherwise  $F_{m,0}(X_{\lambda^{\bullet}, \lambda^{\circ}}) = 0$.
 \item In particular, for any partition $\lambda^{\circ}$ of $n$, we have the corresponding idempotent $e_{\lambda^{\circ}} \in \End_{\mathcal{D}_t}( X_t^{\otimes n})$ whose image $S^{\lambda^{\circ}} X_t$ is the Schur functor $S^{\lambda^{\circ}}$ applied to $X_t$. The object $S^{\lambda^{\circ}} X_t$ is indecomposable, and $e_{\lambda^{\circ}}$ form a complete set of primitive idempotents of $\End_{\mathcal{D}_t}( X_t^{\otimes n})$ for different partitions $\lambda^{\circ}$ of $n$.
\end{itemize}

%

\section{Preliminaries on representations of finite-dimensional Lie superalgebras}\label{sec:rep_superalgebras}
In this section, we remind the reader of some representation theory of Lie superalgebras $ \gl(\infty|\infty), \mathfrak{gl}(m|n)$.
\subsection{Superspaces and Lie superalgebras}

Let $V$ be a super-vector space. We denote by $V_{\bar 0}$ the even part of $V$ and by $V_{\bar 1}$ the odd part of $V$.

Let $ m, n$ be the dimensions of the vector spaces $V_{\bar 0}$, $ V_{\bar 1}$. We denote by $\mathfrak{gl}(m|n)$ the general linear Lie superalgebra of all endomophisms of $V$. 

Similarly, let $\tilde V = \C^{\infty} \oplus \Pi \C^{\infty}$ be the super vector space whose even and odd parts are isomorphic to the countably-dimensional spaces $\C^{\infty} = \bigcup_n \C^n$. We denote by $ \gl(\infty|\infty) = \displaystyle\lim_{\rightarrow}\mathfrak{gl}(m|n)$ 
for the finitary general linear Lie superalgebra of endomophisms of $\tilde V$.

In what follows we denote $\tilde\fg=\mathfrak{gl}(\infty|\infty)$ and 
$\fg=\mathfrak{gl}(m|n)$ for short.

Both Lie algebras $\g, \tg$ are equipped with a $\Z$-grading
$$\tg=\tg_{-1}\oplus\tg_{0}\oplus\tg_{1}, \g=\g_{-1}\oplus\g_{0}\oplus\g_{1}$$
with $\tg_0\simeq \mathfrak{gl}(\infty)\oplus\mathfrak{gl}(\infty)$, $\g_0 \cong \gl(m) \oplus \gl(n)$, $\g_1 \cong V_{\bar 0} \otimes V_{\bar 1}^*, \g_{-1} \cong V_{\bar 0}^* \otimes V_{\bar 1}$.

\subsection{Representations of Lie superalgebras}\label{ssec:rep_superalgebras}
We now describe the category of representations of these superalgebras which we will consider.

The category of (all) representations of any Lie superalgebra $\fg$ is left-tensored
over the category of finite dimensional super vector spaces $\sVect$. This means that, given a super vector space $U$ and a representation $X$, the tensor product
representation $U\otimes X$ is defined. In particular, the parity change functor 
$\Pi$ is the tensor product with the odd one-dimensional super space.

It is often possible to define a smaller category of representations $\rp(\g)$
such that the full category of representations has form $\sVect\otimes\rp(\g)$.

The objects of the category $\sVect\otimes\rp(\g)$ are formal direct sums 
$U\oplus \Pi W$ with $U,W\in\rp(\g)$, with 
\begin{equation}
\Hom(U\oplus\Pi W,U'\oplus\Pi W')=\Hom_{\rp(\fg)}(U,U')\oplus\Hom_{\rp(\fg)}(W,W').
\end{equation}

\begin{remark}
 Such "halved" categories of representations are a special case of Deligne's categories $\rp(G,\epsilon)$ of representations $\sVect\otimes\rp(\g)$ of algebraic groups
in a tensor category, see details in Section \ref{sec:Deligne_conj}.

The category $\rp(\mathfrak{gl}(m|n))$ can be defined as the category of finite-dimensional (super) representations of $\g$, integrable over the algebraic supergroup $GL(m|n)$, on which the action of the element $\id_{V_{\bar 0}} - \id_{V_{\bar 1}} \in GL(m|n)$ is compatible with the grading given by the super structure.
\end{remark}


%

\subsection{Mixed tensor powers}\label{ssec:frrep} 
The Lie superalgebra $\g$ acts on $V$ (this action is called the {\it natural} $\g$-module) and on $V^*=(V^*)_{\bar{0}}\oplus (V^*)_{\bar 1}$ (the {\it conatural} $\g$-module).


 We denote by $T^{p,q}:=V^{\otimes p}\otimes (V^*)^{\otimes q}$ the mixed tensor powers of the natural module. 


 The category $\rp(\g)$ can be described as the full subcategory of category of finite-dimensional (super-)representations of $\g$ which are subquotients of direct sums of $T^{p,q}$, $p, q \geq 0$.
\subsection{Highest weight structure}


The category $\rp(\g)$ has a highest weight structure.

We will use the standard basis of weights in $\g$: consider the standard Cartan subalgebra $\h\subset\g_0$ ($\h \cong \C^m \oplus \C^n$), with basis $\{\varepsilon_1,\ldots,\varepsilon_m, \delta_1, \ldots, \delta_n\}$. We write
weights of modules in $\rp(\g)$ in the form
$$\lambda=(\lambda^{\circ}_1, \ldots, \lambda^{\circ}_m,-\lambda^{\bullet}_1,\ldots,-\lambda^{\bullet}_n)=\sum_{i=1}^m\lambda^{\circ}_i\varepsilon_i - \sum_{i=1}^n\lambda^{\bullet}_i\delta_i,\quad \lambda_i\in\mathbb Z.$$

It will be convenient to us to use a slightly unusual choice of simple roots for $\g$:
$$\varepsilon_1-\varepsilon_2,\dots,\varepsilon_m-\delta_n,\delta_n-\delta_{n-1},\dots,\delta_2-\delta_1.$$

Then the set of dominant weights is given by $$\lambda=\sum_{i=1}^m a_i\varepsilon_i-\sum_{j=1}^nb_j\delta_j$$
for some integers $a_1\geq\dots\geq a_m$, $b_1\geq\dots\geq b_n$.

Simple objects in $\g$ are parametrized up to isomorphism by dominant weights: we denote by $L(\lambda)$ the simple highest weight $\g$-module with dominant highest weight
$\lambda$.


We also have the standard and costandard objects in the highest weight category $\rp(\g)$, again parameterized up to isomorphism by dominant weights. These are called Kac and dual Kac modules respectively, and are denoted by $K(\lambda)$, $\check{K}(\lambda)$. Each simple module $L(\lambda)$ is the unique irreducible quotient of the appropriate Kac module. 

The Kac modules are defined as usual via induction: consider the subalgebra $\g_0\oplus\g_1 \subset \g$, and let $L_{0}(\lambda)$ be the simple $\g_0$-module with highest weight $\lambda$
and trivial action of $\g_1$. Then the induced module $U(\g)\otimes_{U(\g_0\oplus\g_1)}L_{0}(\lambda)$ is defined to be the Kac module $K(\lambda)$. Similarly, the dual Kac module $\check{K}(\lambda)$ can be defined via coinduction.


We call $\lambda$ {\it positive} if $a_i\geq 0$, $b_j\geq 0$ for all $i=1,\dots,m$, $j=1,\dots,n$ and {\it negative} if {\it both} $a_m$ and $b_n$ are negative.

There is a bijection between the set of positive weights and the bipartitions (we denote both by the same symbol), $\lambda=(\lambda^\circ,\lambda^\bullet)$,
with $\lambda^\circ=(a_1,\dots,a_m),\lambda^{\bullet}=(b_1,\dots,b_n)$.

\subsection{Weight diagrams and block decomposition}
Recall the notion of weight diagram $f_\lambda$ associated with any dominant weight $\lambda$. Let $t=m-n$, $c_i=a_i+t-i$, $d_j=b_j-j$. 
The weight diagram corresponding to $\lambda$ is the function $f_\lambda:\mathbb Z\to\{\times,>,<,\circ\}$ defined by
\begin{equation}\label{eq:weightdiagrams}
f_\lambda(s)=\begin{cases} \circ,\,\text{if}\,c_i\neq s,\,d_j\neq s\,\text{for all}\,i,j,\\
>,\,\text{if}\,c_i=s,\,\text{for some}\,i,\,d_j\neq s\,\text{for all}\,j,\\
<,\,\text{if}\,d_j=s,\,\text{for some}\,j,\,c_i\neq s\,\text{for all}\,i,\\
\times,\,\text{if}\,c_i=d_j=s\,\text{for some}\,i,j.
\end{cases}
\end{equation}
Usually a weight diagram is represented by a picture. For instance, if $m=2,n=1$ and $\lambda$ is the highest weight of the standard module,
we have $\lambda=\varepsilon_1$ and $f_\lambda$ is
$$\circ\dots\circ\times\circ>\circ\dots,$$
with $f_\lambda(1)=>$. If $\mu$ is the highest weight of the costandard module, then $\mu=-\delta_1$ and $f_\mu$ is
$$\circ\dots\circ>\times\circ\dots,$$
with $f_\mu(0)=\times$.

The {\it core} diagram $\bar {f}_\lambda$ is obtained from $f_\lambda$ by replacing all $\times$ by $\circ$. The core diagrams enumerate the blocks in 
$\rp(\g)$, since $\chi_\lambda=\chi_\mu$ if and only if $\bar f_\lambda=\bar f_\mu$. So we have a block decomposition
\begin{equation}\label{eq:finblock}
\rp(\g)=\bigoplus_{\chi}\rp(\g)^\chi,
\end{equation}
where summation is taken over all core diagrams $\chi$.

\begin{remark}\label{rem:weight1} It is useful to note that $\lambda$ is positive if and only if $f_\lambda(s)=\circ$ for all $s<-n$. 
\end{remark}

\subsection{On translation functors}
If $\chi$ is a core diagram, we denote by $u(\chi)$ the total 
number of symbols $>$ and $<$ in $\chi$. The degree of atypicality of $\chi$ equals $\frac{m+n-u(\chi)}{2}$.
We will need the following lemma from \cite[Lemma 7]{GS} and \cite[Theorem 3.2]{Ssd}.

\begin{lemma}\label{lem:GS} Let $\g=\mathfrak{gl}(m|n)$ and $\rp(\g)^\chi$ and $\rp(\g)^\theta$ be two blocks associated with core 
diagrams $\chi$ and $\theta$.
Consider the translation functors $T_{\chi,\theta},T^*_{\chi,\theta}:\rp(\g)^\chi\to \rp(\g)^\theta$
%

defined by
$$T_{\chi,\theta}(M)=(M\otimes V)^\theta,\quad T^*_{\chi,\theta}(M)=(M\otimes V^*)^\theta, $$
where $(\cdot)^\theta$ stands for projection onto the block with core $\theta$.

The functors $T_{\chi,\theta}$, $T^*_{\theta,\chi}$ are biadjoint and satisfy the following properties:
\begin{enumerate}[label=(\alph*)]
 \item If $u(\chi)=u(\theta)$ and $\theta$ is obtained from $\chi$ by moving $>$ one position right or $<$ one position left, then
$T_{\chi,\theta}$ defines an equivalence between the abelian categories  $\rp(\g)^\chi$ and $\rp(\g)^\theta$. 

\item If $u(\chi)=u(\theta)$ and $\theta$ is obtained from $\chi$ by moving $<$ one position right or $>$ one position left, then
$T^*_{\chi,\theta}$ defines an equivalence between abelian categories $\rp(\g)^\chi$ and $\rp(\g)^\theta$.   

\item Assume $u(\theta)=u(\chi)-2$ and there exists $s$ such that  $\theta(r)=\chi(r)$ if $r\neq s,s+1$, $\chi(s)=<$, $\chi(s+1)=>$ and
$\theta(s)=\theta(s+1)=\circ$. If $P(\lambda)$ is the projective cover
of $L(\lambda)$ in $\rp(\g)$, then $T^*_{\chi,\theta}(P(\lambda))=P(\mu)$, where 
$$f_\mu(r)=\begin{cases}f_\lambda(r),\,\text{if}\,\, r\neq s,s+1\\
\times,\,\text{if}\,\, r=s\\
\circ,\,\text{if}\,\, r=s+1\end{cases}.$$

\item Assume $u(\theta)=u(\chi)-2$ and there exists $s$ such that  $\theta(r)=\chi(r)$ if $r\neq s,s+1$, $\chi(s)=>$, $\chi(s+1)=<$ and
$\theta(s)=\theta(s+1)=\circ$.
Then $T_{\chi,\theta}(P(\lambda))=P(\mu)$, where 
$$f_\mu(r)=\begin{cases}f_\lambda(r),\,\text{if}\,\, r\neq s,s+1\\
\times,\,\text{if}\,\, r=s\\
\circ ,\,\text{if}\,\, r=s+1\end{cases}.$$

\item Let $u(\chi)\geq u(\theta)$ and $L$ be a simple module in $\rp(\g)^\theta$. Then $T_{\theta,\chi}(L)$ (or $T^*_{\theta,\chi}(L)$) is either simple or
zero. Furthermore, if $T_{\theta,\chi}(L)\simeq T_{\theta,\chi}(L')\neq 0$ (or $T^*_{\theta,\chi}(L)\simeq T^*_{\theta,\chi}(L')\neq 0$) for some simple
$L$ and $L'$ in $\rp(\g)^\theta$, then $L\simeq L'$.

\end{enumerate}
\end{lemma}  

Let us recall that the 
multiplicity of a simple module in a Kac module is at most one and it can be calculated using cap diagrams (see, for instance, \cite{MS}).
We equip $f_\lambda$ with caps following the rule
\begin{enumerate}
\item The left end of a cap is at $\times$ and the right end is at $\circ$;
\item Caps do not overlap;
\item There is no $\circ$ inside a cap \InnaA{which is not the endpoint of some 
other cap};
\item Every $\times$ is the left end of exactly one cap.
\end{enumerate}

Let us recall the following \cite{MS}.
\begin{proposition}\label{prop:multinkac} The multiplicity $[K(\lambda):L(\mu)]$ is at most one. It is $1$ if and only if one can obtain $\lambda$ from
$\mu$ by moving some $\times$ from the left end of its cap to the right end.
\end{proposition}

\subsection{Contragredient duality}
In what follows we also use the contragredient duality functor $\check{}:\rp(\g)\to \rp(\g)$ defined as follows. For any $M\in \rp(\g)$ we set
$$\check M:=(M^*)^\sigma,$$
where the automorphism $\sigma$ of $\g$ is the negative supertransposition. 
We will use the following facts, which are well-known:
\begin{lemma}
 \mbox{}
 \begin{itemize}
  \item If $L\in \rp(\g)$ is simple, then $\check L\simeq L$. Moreover, any highest weight module $L$ such that $\check L$ is isomorphic to $L$
is simple. 
\item The contragredient duality interchanges the Kac module $K(\lambda)$ and the dual Kac module $\check{K}(\lambda)$.
 \end{itemize}

\end{lemma}

\subsection{A filtration on \texorpdfstring{$\rp(\g)$}{Rep(g)}}

Finally, we define a filtration on $\rp(\g)$ by full subcategories $\rp^k(\g)$.
Let $\rp^k(\g)$ be the abelian subcategory of  $\rp(\g)$ consisting of all subquotients of finite direct sums of $\displaystyle\bigoplus_{p+q\leq k}T^{p,q}$.

We call a weight $\lambda$ $k$-{\it admissible} if $L(\lambda)$ belongs $\rp^k(\g)$. A central character $\chi$ is $k$-{\it admissible} if $\chi=\chi_\lambda$
for some $k$-admissible $\lambda$.

Note that $\check{T}^{p,q}\simeq T^{p,q}$, hence  $\rp^k(\g)$ is closed under 
contragredient duality.

\begin{remark}\label{rem:weight2} If a block $\rp(\g)^\chi$
is $k$-admissible for some $k$ then the core diagram $\chi$ must satisfy $\chi(s)=\circ$ for all $s<-n$.
\end{remark}

\section{The category \texorpdfstring{$\T_{\tg}$}{of representations of infinite-rank Lie superalgebra}}\label{sec:infty_rep_prelim}
\subsection{}
Recall that the natural representation $\tV$ of
$\tg$ is countably-dimensional. We will consider its restricted dual $\tV_*$ (also countably-dimensional), and the non-degenerate pairing $ev:\tV \otimes \tV_* \to \C$. The superspace $\tV_*$ obviously carries an action of $\tg$ induced by the action on $\tV$, and is called the conatural representation.
By $\tT^{p,q}$ we denote the $\tg$-module $\tV^{\otimes p}\otimes (\tV_*)^{\otimes q}$, analogue of the mixed tensor power $T^{p,q}$. 

Let $\T_{\tg}$ be the full subcategory of $\tg$-modules consisting of all subquotients of finite direct sums of $\tT^{p,q}$, $p, q \geq 0$. 
This category has an intrinsic characterization: it consists of integrable $\tg$-modules of finite length with prescribed parity
of weight spaces
such that the annihilator of any vector is a finite corank subalgebra in $\tg$, see \cite{Sr}. 

\begin{remark}\label{rem-inv} Note that we have the natural identification $\tg=\tV\otimes\tV_*$. One can see that in contrast with the finite-dimensional case
$\Hom_{\tg}( \mathbb C,\tg)=0$ since $\tg$ does not contain the identity matrix. It is proven in \cite{PS} that in general $\Hom_{\tg}( \mathbb C,\T^{p,q})=0$
if $(p,q)\neq (0,0)$.
\end{remark}
 
 \begin{remark}
  This subcategory is of course not rigid.
 \end{remark}

It is proved  in \cite{Sr} that
$\T_{\tg}$ is equivalent, as a monoidal category, to the similar category $\T_{\mathfrak{gl}(\infty)}$ for the Lie algebra $\mathfrak{gl}(\infty)$. 

This category was also studied by Sam and Snowden and they proved that it is universal in the class of symmetric monoidal categories satisfying additional
properties, see \cite{SS}.

\subsection{Abelian structure}
The module $\bigoplus\tT^{p,q}$ is an injective cogenerator in $\T_{\tg}$. If $\mu$ is a partition of
length $p$, then we denote by $\tV(\mu)$ (respectively, $\tV_*(\mu)$) the image of a Young projector $\pi_\mu$ in $\tT^{p,0}$ (respectively, $\tT^{0,p}$).

It follows from \cite{DPS, Sr} that for any bipartition $\lambda=(\lambda^\circ,\lambda^\bullet)$ of length $(p,q)$, the module 
$\tY(\lambda):=\tV(\lambda^\circ)\otimes\tV_*(\lambda^\bullet)\subset \tT^{p,q}$ is
indecomposable injective in $\mathbb T_{\tilde g}$ with simple socle which we denote by $\tV(\lambda)$. 

It was proved in \cite{DPS} that these modules $\tV(\lambda)$, where 
$\lambda$ is a bipartition, describe all the isomorphism classes of simple objects in $\T_{\tg}$. The module $\tilde{Y}(\lambda)$ will then be the injective hull of $\tV(\lambda)$.

We will use the fact that the simple socle $\tV(\lambda)$ is the intersection of kernels
of all contraction maps
$\tY(\lambda)\to\tT^{p,q}\to \tT^{p-1,q-1}$. Moreover, one can describe the socle filtration of $\tY(\lambda)$,
see \cite{PS}.

\subsection{A filtration on \texorpdfstring{$\T_\tg$}{the category of representations of the infinite-rank Lie algebra}}
By analogy with the finite-dimensional case, we define a filtration on $\T_\tg$ by full subcategories $\T^k_\tg$.
We denote by $\T^k_\tg$  the full abelian subcategory of $\T_\tg$ with injective cogenerator
$\displaystyle\bigoplus_{p+q\leq k}\tT^{p,q}$. 

Clearly, $\T_{\tg}$ is the direct limit  $\displaystyle\lim_{\rightarrow}\T^k_\tg$.


\section{\texorpdfstring{The functor $\mathcal R_{\g}$}{The specialization functor} and the filtration \texorpdfstring{$\rp(\g)$ of $\rp^k(\g)$}{}}\label{sec:specialization_functor}

In this section we describe the specialization functor $\mathcal R_{\g}$ between the categories of representations of $\tg:=\mathfrak{gl}(\infty|\infty)$ and $\g:=\mathfrak{gl}(m|n)$. This functor is left-exact, and takes the natural representation $\tV$ to $V$. We prove that this is an SM functor, which respects the filtrations on $\T_\tg,\rp(\g)$. Furthermore, we show that for $k$ small enough, the functor $ R_{\g}:\T^k_{\tg}\to \rp^k(\g)$ is exact, and use it to describe a highest-weight structure on the subcategories $\rp^k(\g)$.

\subsection{\texorpdfstring{The functor $\mathcal R_{\g}$}{The specialization functor}}
Let $V$ be a $(m|n)$-dimensional subspace of $\tV$ and $W$ be a complementary infinite-dimensional  subspace such that $\tV=V\oplus W$.
Consider the dual decomposition $\tV_*=W^\perp\oplus V^\perp$. Then $W^\perp\simeq V^*$ and $V^\perp\simeq W_*$.
We identify $\g$ with $V\otimes V^*$ and consider the subalgebra $\k\subset \tg$, defined as
$\k:=W\otimes W_*$. Clearly, $\k$ is isomorphic to $\tg$ and $\g$ is the centralizer of $\k$ in $\tg$. 
We define the functor $$\mathcal R_{\g}:\T_\tg\to\rp(\g)$$ by 
$$\mathcal R_{\g}(M):=M^\k,$$ 
the space of $\k$-invariants of $M$.

\begin{lemma}\label{lem:invariants}

\mbox{}

\begin{enumerate}
\item\label{itm:inv1} If $M$ is a submodule in $\tT^{p,q}$, then $\mathcal 
R_{\g}(M)=M\cap T^{p,q}$;
\item\label{itm:inv2} $\mathcal R_{\g}(\T^k_{\tg})\subset \rp^k(\g)$;
\item\label{itm:inv3} $\mathcal R_{\g}$ is a SM functor.
\end{enumerate}
\end{lemma}

\begin{remark}
 As it was shown in \cite{SS}, there is an essentially unique SM left-exact functor $\T_\tg\to\rp(\g)$ taking $\tV$ to $V$. 
\end{remark}

\begin{proof} Consider the decomposition of $\tV$ and $\tV_*$ with respect to $\g\oplus\k$ action. We have
$$\tV=V\oplus W,\quad \tV_*=V^*\oplus W_*.$$
Note that $(W^{\otimes p}\otimes (W_*)^{\otimes q})^\k=0$ for all $p,q\neq 0$ by Remark \ref{rem-inv}. 
$$(\tT^{p,q})^\k=T^{p,q}.$$
Since $\mathcal R_{\g}$ is left exact, we obtain \eqref{itm:inv1}.

The assertion \eqref{itm:inv2} is an immediate consequence of \eqref{itm:inv1}.

To prove \eqref{itm:inv3} it suffices to consider the case $M\subset 
\tT^{p,q},\,N\subset \tT^{r,s}$. Then
$$\mathcal R_{\g} (M\otimes N)=(M\otimes N)\cap T^{p+r,q+s}=(M\cap T^{p,q})\otimes (N\cap T^{r,s})=\mathcal R_{\g}(M)\otimes \mathcal R_{\g}(N).$$
\end{proof}

\begin{theorem}\label{thrm:invariants} 

\mbox{}

\begin{enumerate}
\item\label{itm:thrm_inv1} If $k<\min(m,n)$, then the functor $\mathcal 
R_{\g}:\T^k_{\tg}\to \rp^k(\g)$ is exact.
\item\label{itm:thrm_inv2} Let $2k<\min(m,n)$ and $\lambda$ be a bipartition 
with $|\lambda|:=|\lambda^\circ|+|\lambda^\bullet|\leq k$ and 
$V(\lambda):=\mathcal R_{\g}\tV(\lambda)$. 
Then $V(\lambda)$ is a highest weight module with unique 
irreducible quotient $L(\lambda)$.
\item\label{itm:thrm_inv3} Let $2k<\min(m,n)$, then any simple module in 
$\rp^k(\g)$ is isomorphic to $L(\lambda)$ for some bipartition $\lambda$ with 
$|\lambda|\leq k$.
\end{enumerate}
\end{theorem}
\begin{proof}  Let $\g_0$, $\k_0$ and $\tg_0$ be the even parts of $\g$, $\k$ and $\tg$ respectively. One defines $\T_{\tg_0}$, $\T^k_{\tg_0}$, $\rp(\g_0)$ and $\rp^k(\g_0)$
similarly to the corresponding categories for $\tg$ and $\g$. General results about $\T_{\tg_0}$ can  be found in \cite{Sr}. 
Taking $\k_0$-invariants defines the
SM functor $\mathcal R_{\g_0}: \T_{\tg_0}\to\rp(\g_0)$. A simple direct calculation shows that 
$$\mathcal R_{\g_0}(\tT^{p,q})=T^{p,q}=\mathcal R_{\g}(\tT^{p,q}),$$
hence 
$$\mathcal R_{\g_0}(M)=\mathcal R_{\g}(M)$$
for any $M\in \T_{\tg}$. 
In particular, the following diagram of functors
\begin{equation}
\begin{CD}
\T^k_{\tg}@>\mathcal R_{\g}>>  \rp^k(\g)\\
@VVV @VVV \\
\T^k_{\tg_0}@>\mathcal R_{\g_0}>>\rp^k(\g_0)
\end{CD}\notag
\end{equation}
where the vertical arrows are the restriction functors, is commutative. 

To show \eqref{itm:thrm_inv1} it suffices to prove that $\mathcal R_{\g_0}: 
\T^k_{\tg_0}\to\rp^k(\g_0)$ is exact. Note that 
$$\T^k_{\tg_0}=\bigcup_{r+s=k}\T^r_{\mathfrak{\gl}(\infty)}\boxtimes \T^s_{\mathfrak{\gl}(\infty)},$$
where $\boxtimes$ stands for exterior tensor product and $\mathcal R_{\g_0}=\mathcal R_{\mathfrak{gl}(m)}\boxtimes \mathcal R_{\mathfrak{gl}(n)}$.

We will show first that if $s\leq l$, then $\mathcal R_{\mathfrak{gl}(l)}:\T^s_{\mathfrak{\gl}(\infty)}\to\rp(\mathfrak{gl}(l))^s$ is exact.
Simple objects of $\T^s_{\mathfrak{\gl}(\infty)}$ are of the form $\tV(\lambda)$ for some bipartition
$\lambda$ such that $|\lambda|\leq s$. Furthermore we have $$\mathcal R_{\mathfrak{gl}(l)}(\tV(\lambda))=V(\lambda),$$
and $V(\lambda)$ is a simple $\mathfrak{gl}(l)$-module. Recall that the multiplicity of $V(\lambda)$ in 
$T_{\mathfrak{gl}(\infty)}^{p,q}$
(for $p+q\leq s$) equals the multiplicity of  $V(\lambda)$ in 
$T^{p,q}$, see \cite{PS}. Since $\displaystyle\bigoplus_{p+q\leq s}\tT^{p,q}$ is an injective cogenerator of $\T^s_{\mathfrak{\gl}(\infty)}$,
$\mathcal R_{\mathfrak{gl}(l)}(\tT^{p,q})=T^{p,q}$ and $\mathcal R_{\mathfrak{gl}(l)}$ is left exact, we obtain the statement.
In fact, since $\rp_{\mathfrak{gl}(l)}^s$ is semisimple, the functor $\mathcal R_{\mathfrak{gl}(l)}$ is the semisimplification functor. Then 
$\mathcal R_{\g_0}=\mathcal R_{\mathfrak{gl}(m)}\boxtimes \mathcal 
R_{\mathfrak{gl}(n)}$ is also the semisimplification, and \eqref{itm:thrm_inv1} 
follows.
 
Let us prove \eqref{itm:thrm_inv2}. Recall the decomposition 
$\g=\g_{-1}\oplus\g_0\oplus \g_1$, where $\g_0=\mathfrak{gl}(V_{\bar 
0})\oplus\mathfrak{gl}(V_{\bar 1})$ and
$$\g_{-1}\simeq V_{\bar 0}^*\boxtimes V_{\bar 1}$$ as a $\g_0$-module. Set 
$$U(\lambda):=V_{\bar 0}(\lambda^\circ)\boxtimes V^*_{\bar 1}(\lambda^\bullet).$$
Then $U(\lambda)$ is a simple highest weight  $\g_0$-submodule of $V(\lambda)$ 
and \eqref{itm:thrm_inv2} is equivalent to the fact that
$$V(\lambda)=U(\g_{-1})U(\lambda).$$

Next set $$S^{p,q}:=V_{\bar 0}^{\otimes p}\boxtimes (V_{\bar 1}^*)^{\otimes q}.$$
Let $\phi$ be the direct sum of all contraction maps $T^{p,q}\to T^{p-1,q-1}$. Then by definition $\Ker\phi$ is a direct sum of $V(\lambda)$ with some 
multiplicities and similarly $S(\lambda)$ is a direct sum of 
$U(\lambda)$ with the same multiplicities.
To prove \eqref{itm:thrm_inv2} it suffices to show that for all $p,q$ such that 
$p+q<k$ we have 
$$
U(\g_{-1})S^{p,q}=\Ker \phi,
$$
Note that $U(\g_{-1})$ is isomorphic to the exterior algebra $\Lambda(\g_{-1})$. Define a $\g_0$-map $$\gamma:\Lambda(\g_{-1})\otimes U(\lambda)\to \Ker\phi$$ by setting
$\gamma(x\otimes u):=xu$. It is easy to see that $$\gamma (\Lambda^{p+q+1}(\g_{-1})\otimes S^{p,q})=0.$$ Hence we need to show that
$$\gamma: \Lambda^{\leq p+q}(\g_{-1})\otimes S^{p,q}\to \Ker\phi$$ is surjective.

Let $\tl S^{p,q}:=\tV_{0}^{\otimes p}\boxtimes (\tV_{1}^*)^{\otimes q}$,
$\tilde\phi$  be the direct sum of all contraction maps $\tT^{p,q}\to \tT^{p-1,q-1}$. Consider the map 
$$\tilde\gamma:\Lambda^{\leq p+q}(\tg_{-1})\otimes \tl S^{p,q}\to \Ker\tilde\phi$$
defined in the manner similar to $\gamma$. It is proved in \cite{Sr} that $\tilde\gamma$ is surjective.
Note that both
$\Lambda^{\leq p+q}(\tg_{-1})\otimes \tl S^{p,q}$ and $\Ker\tilde\phi$ are objects of $\T^{2k}_{\tg_0}$, furthermore
$$\mathcal  R_{\g_0}(\Lambda^{\leq p+q}(\tg_{-1})\otimes \tl S^{p,q})=\Lambda^{\leq p+q}(\g_{-1})\otimes S^{p,q},\quad
\mathcal  R_{\g_0}(\tilde\phi)=\phi.$$
Since $\mathcal  R_{\g_0}:\T^{2k}_{\tg_0}\to\rp(\g_0)^{2k}$ is exact  by 
\eqref{itm:thrm_inv1}, we 
obtain that $\mathcal  R_{\g_0}(\Ker\tilde\phi)=\Ker\phi$ and $\gamma=\mathcal  R_{\g_0}(\tilde\gamma)$.
Again the exacteness of  $\mathcal  R_{\g_0}:\T^{2k}_{\tg_0}\to\rp^{2k}(\g_0)$ implies that $\gamma$ is also surjective.

To prove \eqref{itm:thrm_inv3} we will show that under assumption 
$2k\leq\min(m,n)$ any simple subquotient in $V(\lambda)$ with $|\lambda|\leq k$ 
is isomorphic to $L(\mu)$ with
$|\mu|\leq k$. We proceed by induction in $l:=\min(p,q)$, where $p=|\lambda^\circ|,q=|\lambda^\bullet|$. 
Note that the case $l=0$ follows from Sergeev--Schur--Weyl duality.

Consider the submodule $Z(\lambda):=V(\lambda^\circ)\otimes V^*(\lambda^\bullet)$ in  $T^{p,q}$. 
Note that $\check Z(\lambda)\simeq Z(\lambda)$ and $V(\lambda)=Z(\lambda)\cap \Ker\phi$. 
Let $I(\lambda)$ be the sum of images of all coevaluation maps $T^{p-1,q-1}\to T^{p,q}\xrightarrow{p} Z(\lambda)$, here 
$p$ denotes the Young projector on $Z(\lambda)$. 
Consider the contragredient invariant form
$\omega:Z(\lambda)\times Z(\lambda)\to\mathbb C$. Then $V(\lambda)^\perp=I(\lambda)$, hence the restriction of $\omega$ on $V(\lambda)$ has the kernel
$I(\lambda)\cap V(\lambda)$. Therefore  
$V(\lambda)/(I(\lambda)\cap V(\lambda))$ is a contragredient highest weight module, hence it is isomorphic to $L(\lambda)$. 
Since $I(\lambda)$ is a submodule in a direct sum of several copies of $T^{p-1,q-1}$, we conclude by induction assumption that 
all simple subquotients of $I(\lambda)$ are isomorphic to $L(\mu)$ with $|\mu|\leq k$.
That implies the statement for $V(\lambda)$.

\end{proof}
\begin{remark} As follows from Theorem 
\ref{thrm:invariants}\eqref{itm:thrm_inv3}, if $2k<\min(m,n)$ then any 
$k$-admissible weight is positive with additional condition
$a_1+\dots+a_m+b_1+\dots +b_n\leq k$.
\end{remark}

\subsection{Standard modules in subcategories \texorpdfstring{$\rp^k(\g)$}{}}\label{ssec:V_lambda_objects}
In Theorem \ref{thrm:invariants}, we have defined modules $V(\lambda)$ which play the role of standard modules in the subcategories $\rp^k(\g)$. We now describe the actions of translation functors $\otimes V$, $\otimes V^*$ on them, and show that $V(\lambda)$ is the maximal quotient of the Kac module $K(\lambda)$ lying in $\rp^k(\g)$. This will be used in the proof of Theorem \ref{thrm:functorequiv} and in Proposition \ref{prop:presentation}. 

\subsubsection{Action of the translation functors on modules \texorpdfstring{$V(\lambda)$}{V(lam)}}
Recall the following statement from \cite{DPS}.

\begin{lemma}\label{lem:les} Let $\lambda$ be a bipartition. Let $\lambda+\square$ (respectively $\lambda+\blacksquare$) be the set of all bipartitions 
obtained from 
$\lambda$ by adding a box to $\lambda^\circ$ (respectively to $\lambda^\bullet$) and $\lambda-\square$ (respectively $\lambda-\blacksquare$) 
is the set of all bipartitions 
obtained from 
$\lambda$ by removing  a box from $\lambda^\circ$ (respectively from $\lambda^\bullet$).
For any simple object $\tV(\lambda)$ of $\T_\tg$,  there are exact sequences
$$0\to \bigoplus_{\eta\in\lambda+\square}\tV(\eta)\to \tV\otimes \tV(\lambda)\to \bigoplus_{\eta\in\lambda-\blacksquare}\tV(\eta)\to 0,$$
$$0\to \bigoplus_{\eta\in\lambda+\blacksquare}\tV(\eta)\to \tV_*\otimes \tV(\lambda)\to \bigoplus_{\eta\in\lambda-\square}\tV(\eta)\to 0.$$
\end{lemma}
\begin{corollary}\label{cor:tp} If $2k<\min(m,n)-2$ and $V(\lambda)\in\rp^k(\g)$, then there are exact sequences
$$0\to \bigoplus_{\eta\in\lambda+\square}V(\eta)\to V\otimes V(\lambda)\to \bigoplus_{\eta\in\lambda-\blacksquare}V(\eta)\to 0,$$
$$0\to \bigoplus_{\eta\in\lambda+\blacksquare}V(\eta)\to V^*\otimes V(\lambda)\to \bigoplus_{\eta\in\lambda-\square}V(\eta)\to 0.$$
\end{corollary}

\subsubsection{Comparison of \texorpdfstring{$V(\lambda)$}{modules V} with Kac modules}

\begin{proposition}\label{prop:kacrep} If $2k<\min(m,n)$ and $V(\lambda), V(\mu)\in\rp^k(\g)$, then
\begin{enumerate}
\item\label{itm:kacrep1} \InnaB{The module} $V(\lambda)$ is \InnaB{the} maximal quotient of $K(\lambda)$ 
lying in $\rp^k(\g)$ and $\check V(\mu)$ is \InnaB{the} maximal submodule of $\check 
K(\mu)$ in $\rp^k(\g)$;
\item\label{itm:kacrep2} $\dim\Hom (V(\lambda),\check 
V(\mu))=\delta_{\lambda,\mu}$;
\item\label{itm:kacrep3} $\Ext^1 (V(\lambda),\check V(\mu))=0$.
\end{enumerate}
\end{proposition}
\begin{proof} Part \eqref{itm:kacrep1} is a direct consequence of Lemma 
\ref{lem:multiplicities} below and Theorem \ref{thrm:invariants}.

Part \eqref{itm:kacrep2} follows from the fact that $\dim\Hom 
(K(\lambda),\check K(\mu))=\delta_{\lambda,\mu}$ and \eqref{itm:kacrep1}.

Let us prove \eqref{itm:kacrep3}.
Since $V(\lambda)$ is \InnaB{the} maximal quotient of $K(\lambda)$ which belongs
to $\rp^k(\g)$, we have that $V(\lambda)$ is projective in the
Serre subcategory of $\rp^k(\g)$ containing simples $L(\tau)$ for all
$\tau \leq \lambda$. Therefore $\Ext^1 (V(\lambda),\check V(\mu))\neq 0$ implies 
$\lambda<\mu$. On the other hand, $\check V(\mu)$ is a maximal
submodule of $\check K(\mu)$ lying in $\rp^k(\g)$; hence $\check V(\mu)$ is injective in the corresponding Serre subcategory, and therefore $\Ext^1 (V(\lambda),\check V(\mu))\neq 0$ implies
$\mu<\lambda$. Therefore $\Ext^1 (V(\lambda),\check V(\mu))= 0$.
\end{proof}

\begin{lemma}\label{lem:multiplicities} Assume that $2k<\min(m,n)$. Then for any $k$-admissible weights $\lambda$ and $\mu$ we have
$$[K(\lambda):L(\mu)]=[V(\lambda):L(\mu)].$$
\end{lemma} 
\begin{proof}
We will prove the following equivalent statement:
for any $k$-admissible $\lambda$
$$[K(\lambda)]-[V(\lambda)]=\sum [L(\nu)],$$
where all $\nu$ in the righthand side are non-positive (hence not $k$-admissible).

Let $\lambda$ be some $k$-admissible weight, $\Omega$ be the set of all dominant weights of the form $\lambda+\varepsilon_i$ or
$\lambda+\delta_j$ and $\Omega^+$ be the subset of all positive weights in $\Omega$ and $\Omega^-=\Omega\setminus\Omega^+$.
Corollary \ref{cor:tp} implies that
$$[V(\lambda)\otimes V]=\sum_{\mu\in\Omega^+} [V(\mu)].$$
On the other hand,
$$[K(\lambda)\otimes V]=\sum_{\mu\in\Omega} [K(\mu)].$$
If $\mu$ is not positive and $[K(\mu):L(\nu)]\neq 0$, then $\nu\leq \mu$ and therefore $\nu$ is also non-positive.

We prove the statement by induction on $|\lambda|$ starting with the trivial case $\lambda=0$.

Suppose
$$[K(\lambda)]-[V(\lambda)]=\sum c_\nu [L(\nu)].$$
By induction assumption all $\nu$ are non-positive.
Observe that all $\nu$ are negative, since  $\nu$ and $\lambda$ lie in the same block and hence $\bar f_\nu=\bar f_\lambda$, see 
Remark \ref{rem:weight1}.
But, then $L(\nu)\otimes V$ does not have simple subquotients with positive highest weights.
Therefore 
$$[K(\lambda)\otimes V]-[V(\lambda)\otimes V]=\sum c_\nu [L(\nu)\otimes V]$$
also does not have positive terms.
Therefore, if the statement is proved for $V(\lambda)$, it is also true for any $V(\mu)$ which occurs in the tensor product
$V(\lambda)\otimes V$. Similarly it is true for any $V(\mu)$ which occurs in the tensor product
$V(\lambda)\otimes V^*$. Since starting from the trivial module we can obtain any $V(\lambda)$ by tensoring with $V$ and $V^*$ (see Corollary \ref{cor:tp}), the statement follows. 
\end{proof}

\section{Homology functor}\label{sec:homology_functor}
\subsection{}
Recall from \cite{DS} that for any Lie superalgebra $\g$ and an odd element $x\in\g$ such that $[x,x]=0$, one can define
a functor from the category of $\g$-modules to the category of $\g_x$-modules, where $\g_x:=\Ker\ad_x/\Im \ad_x$. This functor sends $M$ to
$M_x:=\Ker x/\Im x$. It is easy to see that this functor is symmetric monoidal: for $\fg$-modules $M,N$ one has a natural isomorphism
$$ M_x\otimes N_x\to (M\otimes N)_x$$
coming from K\"unneth formula, see \cite{Ssd}.

In this paper we are interested in the case when $\g=\mathfrak{gl}(m|n)$
and $x\in\g_1$ is a matrix of rank $1$. Then $\g'=\g_x$ is isomorphic to $\mathfrak{gl}(m-1|n-1)$. 
To see this, use the identification $\g=V\otimes V^*$, choose $x=v\otimes\varphi$ for some $v\in V$ and $\varphi\in V^*$.
Then $\g_x$ can be identified with $\varphi^\perp\otimes v^\perp\simeq \mathfrak{gl}(m-1|n-1)$.
We denote the corresponding functor by $\mathcal H$.

In what follows we denote by $V'$, $V'(\lambda)$ and $(T^{p,q})'$ the analogues of $V$, $V(\lambda)$ and $T^{p,q}$ in $\rp(\g')$.
By a simple calculation one can see that
$\mathcal H(V)\simeq V'$ and $\mathcal H(V^*)\simeq (V')^*$, and hence $$\mathcal H(T^{p,q})=(T^{p,q})'.$$

It is easy to see that the restriction functor maps $\rp^k(\g)$ to $\sVect\otimes\rp^k(\g')$. Hence
the restriction of $\mathcal H$ defines the functor $\mathcal H:\rp^k(\g)\to \sVect\otimes\rp^k(\g')$.
Our next goal is to prove the following result.
 
\begin{theorem}\label{thrm:functorequiv} If $4k<\min(m,n)$, then  $\mathcal H$ maps $\rp^k(\g)$ to $\rp^k(\g')$ and establishes an equivalence of these categories.
\end{theorem}
  
The proof will be done in several steps. 
First, let us prove the following general statement.

\begin{lemma}\label{lem:auxequiv} Consider an exact sequence
$$0\to A\to B\to C\to 0$$
in the category $\rp(\g)$, and the sequence in $\rp(\g')$ 
\begin{equation}\label{seq2}
0\to \mathcal H(A)\to \mathcal H(B)\to \mathcal H(C)\to 0
\end{equation}
obtained by application of $\mathcal H$. 
\begin{enumerate}
\item If $\mathcal H(A)\to \mathcal H(B)$ is an injection or  $\mathcal H(B)\to \mathcal H(C)$ is a surjection, then
(\ref{seq2}) is exact.
\item For any simple $S$ in $\rp(\g')$ we have
$$[\mathcal H(B):S]-[\mathcal H(B):\Pi S]=[\mathcal H(A):S]+[\mathcal H(C):S]-[\mathcal H(A):\Pi S]-[\mathcal H(B):\Pi S].$$
\end{enumerate}
\end{lemma}
\begin{proof} Note that all modules in $\rp(\g)$ have a $\mathbb Z$-grading compatible with the canonical $\mathbb Z$-grading of $\g$. Hence $(A,x)$, $(B,x)$ and
$(C,x)$ can be considered as complexes. Then the statement is a direct conseqence of the long exact sequence of cohomology.
\end{proof}

\begin{lemma}\label{lem:simple} If $2k<\min(m,n)$, then
\begin{enumerate}
\item\label{itm:simple1} $\mathcal H(V(\lambda))\simeq V'(\lambda)$ for any 
$V(\lambda)$ in $\rp^k(\g)$,
\item\label{itm:simple2} $\mathcal H(L(\lambda))\simeq L'(\lambda)$ for any 
simple $L(\lambda)$ in $\rp^k(\g)$,
\item\label{itm:simple3} $\mathcal H(\rp^k(\g))\subset \rp^k(\g')$ and 
$\mathcal H:\rp^k(\g)\to \rp^k(\g')$ is an exact functor.
\end{enumerate}
\end{lemma}
\begin{proof} First, we observe that \eqref{itm:simple1} follows easily from 
Corollary \ref{cor:tp} and Lemma \ref{lem:auxequiv} since $\mathcal H$ is an SM 
functor.

We prove \eqref{itm:simple2} and \eqref{itm:simple3} by induction on $k$ 
assuming that both statements are true for $s<k$. For $k=0,1$ the first 
statement is trivial and the second follows
from semisimplicity of the involved categories. 

Consider the exact sequence
$$0\to I(\lambda)\xrightarrow{\tau} V(\lambda)\xrightarrow {\sigma} L(\lambda)\to 0$$
in $\rp^k(\g)$ and the similar exact sequence
$$0\to I'(\lambda)\xrightarrow{\tau'} V'(\lambda)\xrightarrow{\sigma'} L'(\lambda)\to 0$$
in $\rp^k(\g')$.
It follows from the construction of $I(\lambda)$ in the proof of Theorem \ref{thrm:invariants},
that $\tau'=\mathcal H(\tau)$. By the induction assumption we have 
$\mathcal H(I(\lambda))\simeq I'(\lambda)$ and Lemma \ref{lem:auxequiv} 
ensures that
$$0\to I'(\lambda)\xrightarrow{\tau'} V'(\lambda)\xrightarrow{\sigma'}\mathcal H(L'(\lambda))\to 0$$
is exact. Therefore $\mathcal H(L(\lambda))\simeq L'(\lambda)$.

Now, when we have \eqref{itm:simple2} for $\rp^k(\g)$, \eqref{itm:simple3} 
follows by induction on length and Lemma~\ref{lem:auxequiv}.
\end{proof}

\begin{lemma}\label{lem:ext} Let $2k<\min{(m,n)}$. If an exact sequence 
$$0\to L(\lambda)\to M\to L(\mu)\to 0$$
in $\rp^k(\g)$ does not split, then the exact sequence
$$0\to L'(\lambda)\to \mathcal H(M)\to L'(\mu)\to 0$$
does not split in  $\rp^k(\g')$.
\end{lemma} 
\begin{proof} We have only two possibilities, $\lambda<\mu$ or $\mu<\lambda$. In 
the former case, Proposition \ref{prop:kacrep}\eqref{itm:kacrep1} and the 
assumption of the lemma implies
that $M$ is a highest weight module with highest weight $\mu$, and thus is a quotient of $V(\mu)$. Hence $\mathcal H(M)$ is a 
quotient of $V'(\mu)=\mathcal H(V(\mu))$ which is indecomposable. The second case can be reduced to the first one using duality.
\end{proof}

\begin{corollary}\label{cor:socle} Let $M$ be an object in $\rp^k(\g)$ with $2k\leq \min(m,n)$. Then the natural map
$$\phi:\mathcal H(\operatorname{soc}M)\to \operatorname{soc}(\mathcal H(M))$$
is an isomorphism.
\end{corollary}
\begin{proof} By exactness of $\mathcal H$ and Lemma \ref{lem:simple}, we know that $\phi$ is injective.
Surjectivity of $\phi$ is a consequence of Lemma \ref{lem:ext}.
\end{proof}

\begin{corollary}\label{cor:inv} For all $M\in \rp^k(\g)$ with $2k\leq \min(m,n)$ we have an isomorphism 
$$\operatorname {Hom}_\g(\mathbb C,M)\simeq \operatorname{Hom}_{\g'}(\mathbb C,\mathcal H(M)).$$
\end{corollary}
\begin{proof} We note that
$$\operatorname {Hom}_\g(\mathbb C,M)=\operatorname {Hom}_\g(\mathbb C,\operatorname{soc}M).$$
Hence the statement follows from Corollary  \ref{cor:socle} and Lemma 
\ref{lem:simple}\eqref{itm:simple2}.
\end{proof}

\begin{corollary}\label{cor:ff} If $4k<\min(m,n)$, then the functor $\mathcal H:\rp^k(\g)\to\rp^k(\g')$ is fully faithful.
\end{corollary}
\begin{proof} Note that if $M,N\in \rp^k(\g)$, then $\Hom_\g(M,N)\simeq \Hom_\g(\mathbb C,M^*\otimes N)$.
Therefore, the statement is a direct consequence of Corollary \ref{cor:inv}.
\end{proof}

\begin{lemma}\label{lem:essur} If $4k<\min(m,n)$, then the functor $\mathcal H:\rp^k(\g)\to\rp^k(\g')$ is essentially surjective.
\end{lemma}
\begin{proof} We have proved that  $\mathcal H$ is exact and fully faithful and establishes bijection on the isomorphism classes of simple modules.
Therefore for any $M\in \rp^k(\g)$ and any submodule or quotient $N'$ of $\mathcal H(M)$
there exists a submodule (resp. quotient) $N$ of $M$ such that $N'=\mathcal H(N)$. Since every $M'\in\rp^k(\g')$ is a subquotient of
$T'=\bigoplus (T^{p_i,q_i})'$ and $T'=\mathcal H(\bigoplus T^{p_i,q_i})$, the statement follows.
\end{proof}

Corollary \ref{cor:ff} and Lemma \ref{lem:essur} imply that $\mathcal H:\rp^k(\g)\to\rp^k(\g')$ is an equivalence. 
The proof of Theorem \ref{thrm:functorequiv} is complete.

\subsection{Compatability of specialization and homology functors}\label{ssec:homology}


Recall now functors $\mathcal R_\g:\T_\tg\to \rp(\g)$ and $\mathcal R_{\g'}:\T_\tg\to \rp(\g')$.
We claim that there is a morphism of functors $\Psi:\mathcal R_{\g'}\to \mathcal H\circ\mathcal R_{\g}$.
Let $\k$ and $\k'$ be the centralizers in $\tg$ of $\g$ and $\g'$ respectively. Then $\k\subset\k'$, $x\in\k'$ and $\k=\k'_x$.
For any $M\in\T_{\tg}$ we set $\Psi_M$ to be the composition
$$\Psi_M:M^{\k'}\hookrightarrow (M^\k)^x\rightarrow (M^\k)^x/(xM\cap M^\k).$$

This defines a $\otimes$-natural transformation of SM functors
$$\Psi:\mathcal{R}_{\g'}\longrightarrow \mathcal{H}\circ\mathcal{R}_{\g}, \; \; \Psi_M: R_{\g'}(M) \rightarrow  H\circ\mathcal R_{\g}(M)$$

\begin{lemma}\label{lem:equiv1} 
The restriction of $\Psi:\mathcal{R}_{\g'}\to \mathcal{H}\circ\mathcal{R}_{\g}$
to $\T^k_\tg$ is an isomorphism for $2k<\min(m,n)$.
\end{lemma}
\begin{proof} Let $2k<\min(m,n)$.

By Theorem \ref{thrm:invariants} and Lemma \ref{lem:simple}, the restrictions of both functors $\mathcal{R}_{\g'},  \mathcal{H}\circ\mathcal{R}_{\g}$ to $\T^k_\tg$ are exact functors. Therefore it is enough to prove the required statement for simple objects in $\T^k_\tg$, and then use induction on the length of objects.

We prove that $\Psi_{\tV(\lambda)}$ is an isomorphism by induction on $s=|\lambda|$. For $|\lambda|=0$, the statement is obvious.
Assume that $\Psi_{\tV(\lambda)}:V'(\lambda)\to \mathcal H(V(\lambda))$ is an isomorphism.
Then $\Psi_{\tV(\lambda)\otimes\tV}:V'(\lambda)\otimes V'\to \mathcal 
H(V(\lambda)\otimes V)$
is also an isomorphism since all involved functors are SM. 
Consider the exact sequence
$$0\to \bigoplus_{\eta\in\lambda+\square}\tV(\eta)\to \tV\otimes \tV(\lambda)\to \bigoplus_{\eta\in\lambda-\blacksquare}\tV(\eta)\to 0.$$
Applying $\mathcal R_{\g'}$ to it we obtain the exact sequence
$$0\to \bigoplus_{\eta\in\lambda+\square}V'(\eta)\to V'\otimes V'(\lambda)\to \bigoplus_{\eta\in\lambda-\blacksquare}V'(\eta)\to 0,$$
applying $ \mathcal H\circ\mathcal R_{\g}$ and  using the induction assumption we get
$$0\to \bigoplus_{\eta\in\lambda+\square}\mathcal H\circ\mathcal R_{\g}(\tV(\eta))\to V'\otimes V'(\lambda)\to \bigoplus_{\eta\in\lambda-\blacksquare}V'(\eta)\to 0.$$
By Lemma \ref{lem:auxequiv} the latter sequence is also exact, which implies
the isomorphism $$\Psi_{V'(\eta)}:V'(\eta)\to\mathcal H\circ\mathcal R_{\g}(\tV(\eta))$$ for all $\eta\in\lambda+\square$. 
Repeating the same argument for tensor product
with $\tV_*$ gives an isomorphism  $V'(\eta)\simeq\mathcal H\circ\mathcal R_{\g}(\tV(\eta))$ for all $\eta\in\lambda+\blacksquare$. Hence we have the statement 
for all $\eta$ such that $|\eta|=s+1$.

\end{proof}

\section{The \texorpdfstring{category $\cV_t$}{new category}}\label{sec:maincategory}
\subsection{A new tensor category}\label{ssec:newtc}
Now we fix an integer $t\in\mathbb Z$ and consider all pairs of
non-negative integers $(m,n)$ with $m-n=t$. We fix $x$ in each $\mathfrak{gl}(m|n)$ and consider SM  functors  
$$\mathcal H_{m,n}:\rp^k(\mathfrak{gl}(m|n))\to \rp^k(\mathfrak{gl}(m-1|n-1))$$
defined as in the previous subsection. Consider the inverse limit
$$\cV^k_t=\lim_{\leftarrow}\rp^k(\mathfrak{gl}(m|n)).$$
Theorem \ref{thrm:functorequiv} implies that $\cV^k_t$ is an abelian category; furthermore it
is equivalent to $\rp^k(\mathfrak{gl}(m|n))$ for sufficiently large $m$ and $n$.

Now observe that for any $k<l$ we have an embedding of abelian categories $\cV^k_t\subset \cV^l_t$, since we have such an embedding 
$\rp^k(\mathfrak{gl}(m|n))\subset \rp^l(\mathfrak{gl}(m|n))$ for 
any $m$ and $n$. So we define a new abelian category
$$\cV_t=\lim_{\rightarrow} \cV^k_t.$$

Next, observe that we have bifunctor $\cV^k_t\times \cV^l_t\to \cV^{k+l}_t $ given by tensor product  
$$\rp^k(\mathfrak{gl}(m|n))\times\rp^l(\mathfrak{gl}(m|n))\to\rp^{k+l}(\mathfrak{gl}(m|n))$$
for sufficiently 
large $m,n$. Therefore, passing to direct limit we get a bifunctor $\cV_t\times \cV_t\to \cV_t$ and the following is straightforward:
\begin{lemma}\label{lem:tnsorabelian} The category $\cV_t$ is a tensor category.
\end{lemma} \qed

We denote by $V_t$ the inverse limit of the natural objects for $\g=\mathfrak{gl}(m|n)$, such that $m-n=t$. One can immediately see that $\dim(V_t) =t$.

By the universal property of Deligne's category $\cD_t$, there is a unique,
up to unique $\otimes$-isomorphism, SM functor $I:\cD_t\to\cV_t$, carrying the generator $X_t$ to $V_t$.
 
\begin{proposition}\label{prop:deligneconnection} The functor $I:\cD_t\to\cV_t$
is fully faithful.
\end{proposition}
\begin{proof} Let $X_t^{p,q}=X_t^{\otimes p}\otimes X_t^{*\otimes q}$. 
It is proved in \cite{BS} 
that if $\g=\mathfrak{gl}(m|n)$, $t=m-n$, then $\operatorname{End}(X_t^{p,q})\to \operatorname{End}(T^{p,q})$ is an isomorphism for 
sufficiently large $m$ and $n$. Note that $X_t^{r,s}, X_t^{r',s'}$ (respectively, $T^{r,s}, T^{r',s'}$) can be realized as direct summands in $X_t^{p,q}$
(respectively, $T^{p,q}$) for a suitable choice of $p,q$, so we also have that   
$\operatorname{Hom}(X_t^{r,s}, X_t^{r',s'})\to \operatorname{Hom}(T^{r,s},T^{r',s'})$ are isomorphisms for sufficiently large $m,n$. 
Therefore, Theorem \ref{thrm:functorequiv} implies that $I: \cD^k_t\to \cV^k_t$ is fully faithful, and hence $I: \cD_t\to \cV_t$ is also fully faithful 
by passing to direct limit.
\end{proof}

Now fix $m$ and $n$ such that $m-n=t$ and consider the SM functor $\mathcal H_{m,n}^{s}=\mathcal H_{m+1,n+1}\circ\dots\circ \mathcal H_{m+s,n+s}$ from
$\rp(\mathfrak{gl}(m+s|n+s))$ to $\rp(\mathfrak{gl}(m|n))$. Then we define functor  $F_{m,n}:\cV_t\to \rp(\mathfrak{gl}(m|n))$ by setting 
$F_{m,n}(M)= H_{m,n}^{s}(M)$ for sufficiently large $s$. It follows from Theorem \ref{thrm:functorequiv} that $H_{m,n}^{s}(M)$ stabilizes.

\begin{lemma}\label{lem:functorgl} The functor $F_{m,n}:\cV_t\to \rp(\mathfrak{gl}(m|n))$ is a SM functor and the composition 
$F_{m,n}\circ I: \cD_t\to \rp(\mathfrak{gl}(m|n))$ is the functor defined uniquely up to isomorphism by universality of $\cD_t$.
\end{lemma}
\begin{proof} Straightforward consequence of Theorem \ref{thrm:functorequiv}.
\end{proof}

Recall now the category $\T_\tg$ and the functor $\mathcal{R}_{\mathfrak{gl}(m|n)}:\T_\tg\to \rp(\mathfrak{gl}(m|n))$. 
Lemma \ref{lem:equiv1} implies that for a fixed $k$ and sufficiently large $m,n$ we have the canonical isomorphism
of functors  $$\mathcal{R}_{\mathfrak{gl}(m-1|n-1)}:\T^k_\tg\to \rp^k(\mathfrak{gl}(m-1|n-1))$$ and
$$\mathcal H\circ\mathcal{R}_{\mathfrak{gl}(m|n)}:\T^k_\tg\to \rp^k(\mathfrak{gl}(m-1|n-1)).$$ 
Hence we can define a functor $\Phi^k:\T^k_\tg\to\cV^k_t$  as the inverse limit 
$\displaystyle\lim_{\leftarrow}\mathcal R_{\mathfrak{gl}(m|n)}$, and therefore
$\Phi:\T_\tg\to\cV_t$ by passing to direct limit.

\begin{lemma}\label{lem:functortg} The functor $\Phi$ is a SM functor. Furthermore, $\Phi$ is exact, $\Phi(\tV(\lambda))=V_t(\lambda)$
and $ F_{m,n} \circ \Phi=\mathcal R_{\mathfrak{gl}(m|n)}$. 

\end{lemma}
\begin{proof} First, $\Phi$ is a SM functor since $\mathcal R_{\mathfrak{gl}(m|n)}$ is SM. The exactness of $\Phi$ follows from exactness
of restriction $\mathcal R_{\mathfrak{gl}(m|n)}:\T^k_\tg\to 
\rp^k(\mathfrak{gl}(m|n))$ for sufficiently large $m$ and $n$, see Theorem 
\ref{thrm:invariants}\eqref{itm:thrm_inv1}. 
The identity $\Phi(\tV(\lambda))=V_t(\lambda)$
is a direct consequence of  Theorem \ref{thrm:invariants}. Finally, the last assertion follows from the definition of $\Phi$ as the inverse limit.  
\end{proof}
\begin{corollary}
 For any injective object $E \in \T_\tg$, we have: $\Phi(E) \in I(\cD_t)$.
\end{corollary}

\begin{proof}
 The full subcategory of injective objects in $\T_\tg$ is the full Karoubian symmetric monoidal subcategory of $\T_\tg$ generated by the objects $\tV, \tV_*$. Similarly, $I(\cD_t)$ is the full Karoubian symmetric monoidal subcategory of $\cV_t$ generated by the objects $$V_t = \Phi(\tV), \; V^*_t = \Phi(\tV_*)$$ This immediately implies the desired statement.
\end{proof}

\subsection{Properties \texorpdfstring{of the category $\cV_t$}{}}\label{ssec:properties_cat_V_t}
We now list several "local" and "global" properties of the categories $\cV_t$. 

The local properties are properties of the subcategories $\cV_t^k$, and follow quite easily from the fact that $\cV_t^k$ is equivalent to $\rp^k(\mathfrak{gl}(m|n))$ for sufficiently large $m$ and $n$ such that $m-n =t$. 

The subcategories $\cV^k_t$ satisfy the following properties:
\begin{enumerate}
\item\label{itm:prop_Vt_1} Simple objects in $\cV^k_t$ are enumerated by 
bipartitions $\lambda$ with $|\lambda|\leq k$. Every simple object
is isomorphic to $L_t(\lambda)$ which we define as the inverse limit of simple $\mathfrak{gl}(m|n)$-modules $L(\lambda)$;
\item\label{itm:prop_Vt_2} Any object in $\cV^k_t$ has finite length;
\item\label{itm:prop_Vt_3} The category $\cV^k_t$ has enough projectives and 
injectives.
\item\label{itm:prop_Vt_4} The contragredient duality functor $\ 
\check{}:\rp^k(\mathfrak{gl}(m|n))\to \rp^k(\mathfrak{gl}(m|n))$ extends to the 
corresponding functor $\check{}:\cV^k_t\to\cV^k_t$;
\item\label{itm:prop_Vt_5} For any bipartition $\lambda$ with $|\lambda|\leq k$ 
we define $V_t(\lambda)$ as the inverse limit of $V(\lambda)$. Then 
the cosocle of $V_t(\lambda)$ and the socle of $\check V_t(\lambda)$ are isomorphic to
$L_t(\lambda)$; 
\item\label{itm:prop_Vt_6} The tensor structure on $\cV_t$ is given by maps 
$\cV_t^k \otimes \cV_t^l \rightarrow \cV^{k+l}$ and $\cV_t^k$ is closed under 
the tensor duality contravariant functor $(\cdot)^*$

\end{enumerate}


The only non-trivial statement in the above list is the existence of enough projective and injective objects in $\cV^k_t$ (equivalently, in $\rp^k(\gl(m|n))$). The existence of projective objects (and by duality, injective objects) can be seen as follows:

Consider the inclusion functor $$j^k: \rp^k(\gl(m|n)) \hookrightarrow \rp(\gl(m|n))$$
This functor has adjoints on both sides, its left adjoint $j^k_!$ being the functor which takes a $\gl(m|n)$-module to its maximal quotient which lies in $\rp^k(\gl(m|n))$.

This formally implies that this functor takes projective modules in $ \rp(\gl(m|n))$ to projective objects in $\rp^k(\gl(m|n))$; since $ \rp(\gl(m|n))$ has enough projectives, so does $\rp^k(\gl(m|n))$.


The category $\cV_t$ satisfies a similar list of "global" properties:

The simple objects of $\cV_t$ are $L_t(\lambda)$ for all bipartitions $\lambda$ (the restriction on the size of bipartition disappears) and properties
\eqref{itm:prop_Vt_2}, \eqref{itm:prop_Vt_4}, \eqref{itm:prop_Vt_5} hold. One 
should stress that $\cV_t$ does not have projective nor injective objects.
%
%
%
%

We conclude with a lemma which will be useful later; this lemma is a "local" analogue of the fact that given a tensor category, 
a projective object $P$ and any object $X$, the object $P \otimes X$ is once again projective.

Denote by $$\mathit{i}^k: \cV_t^k \hookrightarrow \cV_t$$ the inclusion functor. This functor has left and right adjoints; its left adjoint $\mathit{i}^k_!$ is the functor which takes an object of $\cV_t$ to its maximal quotient lying in $\cV_t^k$.
\begin{lemma}\label{lem:loc_proj_almost_ideal}
 Let $k \geq l \geq 0$. Let $P$ be a projective object in $\mathcal{V}_t^k$, and let $X$ be any object in $\mathcal{V}_t^l$. Then $\mathit{i}^{k-l}_!(P \otimes X)$ is a projective object in $\mathcal{V}_t^{k-l}$.
\end{lemma}
\begin{proof}
 Denote $Y := \mathit{i}^{k-l}_!(P \otimes X)$. Then we have isomorphisms of functors $\mathcal{V}_t^{k-l} \rightarrow \mathtt{Vect}$
 $$\Hom_{\mathcal{V}_t^{k-l}}(Y, (\cdot)) \cong \Hom_{\mathcal{V}_t}(P \otimes X, (\cdot))\cong \Hom_{\mathcal{V}_t^k}(P, X^* \otimes (\cdot))$$
 
 Using the fact that $P$ is projective in $\mathcal{V}_t^k$, we conclude that the functor $\Hom_{\mathcal{V}_t^k}(P, X^* \otimes (\cdot) )$ on $\mathcal{V}_t^{k-l}$ is exact.
\end{proof}

\subsection{Infinite weight diagrams, blocks and translation functors} 
 In this subsection we describe the block decomposition of the category $\cV_t$. One can immediately see that the blocks of $\cV_t$ correspond to 
the blocks of $\mathcal{D}_t$ classified in \cite{CW}.

For the next result we need the analogues of weight diagrams and translation functors for the category $\cV_t$.
Let $t\in\mathbb Z$ and $\lambda=(\lambda^\circ, \lambda^\bullet)$ be a bipartition. We associate to $\lambda$ two infinite sequences
$c_1,c_2,\dots,$ and $d_1,d_2,\dots$ defined by $c_i=\lambda^\circ_i+t-i$, $d_i=\lambda^\bullet_i-i$. Here we assume that $\lambda^\circ_i=\lambda^\bullet_i=0$ for
sufficiently large $i$. We define the weight diagram
$d_\lambda$ associated to $\lambda$ by the same rule as in 
\eqref{eq:weightdiagrams}. The only difference with $f_\lambda$ is that now 
our sequences are infinite
and hence $d_\lambda(s)=\times$ for all $s<<0$. 

For example, let $t=2$ and $\lambda$ is an empty  bipartition, then the corresponding weight diagram is
$$\dots\times\times >>\circ\circ\dots$$
with $>>$ at the position $0$ and $1$. 
If $\lambda^\circ=(1)$ and $\lambda^\bullet=(1)$ then $d_\lambda$ is of the form
$$\dots\times\times >\times\circ >\circ\circ\dots$$
with rightmost $>$ at $2$.
Weight diagrams associated to bipartition always have finitely many symbols $>$ and $<$ which we call
{\it core} symbols. It is easy to see that $t$ equals the difference between the number of $>$ and the number of $<$ in $d_\lambda$.
Furthemore, by construction all sufficiently large positive positions 
are empty. 

The core $\bar{d}_\lambda$ of the weight diagram $d_\lambda$ is obtained by removing all $\times$ and replacing them by $\circ$.
In the first example $\bar{d}_\lambda$ is
$$\dots\circ\circ >>\circ\circ\dots$$
and in the second
$$\dots\circ\circ >\circ\circ >\circ\circ\dots$$

The following is straightforward.
\begin{lemma}\label{lem:diagrams} Recall the equivalence between the categories $\cV^k_t$ and $\rp^k(\g)$, where $\g={\mathfrak{gl}(m|n)}$ for sufficiently large 
$m,n$ such that $m-n=t$. Assume that $\lambda$ is a bipartition such that $L(\lambda)\in \rp^k(\g)$ and denote $f_\lambda$ the corresponding weight diagram.
Then
\begin{enumerate}
\item $f_\lambda(s)=\circ$ for $s<-n$;
\item $d_\lambda(s)=\times$ for $s<-n$;
\item $d_\lambda(s)=f_\lambda(s)$ for $s\geq -n$.
\item $\bar d_\lambda=\bar f_\lambda$.
\end{enumerate}
\end{lemma}
\qed

Next we define the cap diagram associated to a given weight diagram $d_\lambda$ following the same rule as in finite dimensional case.
Note that in this case a cap diagram has infinitely many caps.

\begin{lemma}\label{lem:multmain} In the category $\cV_t$ we have
$[V(\lambda):L(\mu)]\leq 1$. Furthemore,
$[V(\lambda):L(\mu)]=1$ if and only if $d_\lambda$ is obtained from $d_\mu$ by moving finitely many crosses from the left end of its cap to the right end. 
\end{lemma}
\begin{proof} Follows from Lemma \ref{lem:multiplicities} and Proposition \ref{prop:multinkac}.
\end{proof}

\begin{lemma}\label{lem:infblocks} 
(a) The category $\cV_t$ has a block decomposition. Two simple objects $L_t(\lambda)$ and $L_t(\mu)$ are in the same block
if and only if $\bar{d}_\lambda=\bar{d}_\mu$. Thus,
$$\cV_t=\bigoplus_{\chi} \cV_t^\chi,$$
where $\chi$ runs the set of all possible core diagrams and the simple objects of $ \cV_t^\chi$ are isomorphic to $L_t(\lambda)$
with $\bar{d}_\lambda=\chi$.

(b) Recall the SM functor $F_{m,n}:\cV_t\to \rp(\mathfrak{gl}(m|n))$. Then $F_{m,n}(\cV_t^\chi)$ is a subcategory in $\rp(\g)^\chi$.

(c) Define the translation functors 
$\overline{T}_{\theta,\chi}, \overline{T}^*_{\theta,\chi}:\cV_t^\theta\to \cV_t^\chi$
by
$$\overline{T}_{\theta,\chi}(M)=(M\otimes V_t)^\theta,\quad \overline{T}^*_{\theta,\chi}(M)=(M\otimes V_t^*)^\theta.$$
Then
$$F_{m,n}\circ \overline{T}_{\theta,\chi}={T}_{\theta,\chi}\circ F_{m,n},\quad F_{m,n}\circ \overline{T}^*_{\theta,\chi}={T}^*_{\theta,\chi}\circ F_{m,n}.$$
As before, the functors $\overline{T}_{\theta,\chi}$ and $\overline{T}^*_{\chi,\theta}$ are biadjoint for every $\chi, \theta$.
\end{lemma}
\begin{proof} Choose $k$ such that both $L(\lambda)$ and $L(\mu)$ belong to $\cV^k_t$ and $m,n$ such that  $\cV^k_t$ is equivalent to
$\rp^k(\mathfrak{gl}(m|n))$. Then, the core diagrams are the same for  $\cV^k_t$ and $\rp^k(\mathfrak{gl}(m|n))$ and hence all assertions are straightforward.
\end{proof}

\begin{remark}\label{rem:inftranslation} It follows from above lemma that Lemma \ref{lem:GS} (a),(b),(e) holds for 
$\overline{T}_{\theta,\chi}$ and $\overline{T}^*_{\theta,\chi}$ if one uses
$d_\lambda$ instead of $f_\lambda$.
\end{remark}

\subsection{Existence of presentation}
The goal of this section is to prove the following statement, which is crucial in the proof of universality of $\cV_t$:

\begin{proposition}\label{prop:presentation} 
For any object $M \in \cV_t$ there exists a presentation $$ T' \twoheadrightarrow M \hookrightarrow T''$$ where $T', T''$ are in the image of $\mathcal{D}_t$.
\end{proposition}
\begin{proof} 

Due to the existence of the duality contravariant functor $\check{(\cdot)} : \cV_t \rightarrow \cV_t$, it is enough to show that for any object 
$M \in \cV_t$ there exists an epimorphism $T \twoheadrightarrow M $ where $T$ is in the image of $\mathcal{D}_t$.

We will prove the statement in several steps.

\begin{enumerate}
 \item[Step 1]: We prove the statement for $M= P_k(\emptyset)$, the projective cover of $\triv$ in $\cV_t^k$.
 \item[Step 2]: We prove the statement for any standard object $V_t(\lambda)$ in $\cV_t$.
 \item[Step 3]: We prove the statement for any projective object $P$ in $\cV_t^k$, for any $k \geq 0$.
 \item[Step 4]: We prove the statement for any object $M$ in $\cV_t$.
\end{enumerate}
Steps 1 and 2 are two independent special cases of the general statement, and will be proved in Lemmas \ref{lem:trivialprojective} and \ref{lem:simplequotient}.

{\bf Step 3}: It is enough to prove the statement for indecomposable objects $P_k(\lambda)$ (the projective cover of $L_t(\lambda)$ in $\cV_t^k$), where $|\lambda| \leq k$.

Let $$Y:= \mathit{i}^k_!( P_{2k}(\emptyset) \otimes L_t(\lambda) )$$ be the maximal quotient of $P_{2k}(\emptyset) \otimes L_t(\lambda)$ lying in $\cV_t^k$ (c.f. Subsection \ref{ssec:properties_cat_V_t}). 

By Lemma \ref{lem:loc_proj_almost_ideal}, $Y$ is a projective object in $\cV_t^k$. The covering epimorphism $ P_{2k}(\emptyset) \twoheadrightarrow \triv$ induces an epimorphism $$P_{2k}(\emptyset) \otimes L_t(\lambda) \twoheadrightarrow L_t(\lambda) $$ which factors through an epimorphism $Y \twoheadrightarrow L_t(\lambda) $.

By definition of projective cover, the latter induces a split epimorphism $$Y \twoheadrightarrow P_k(\lambda) $$

We now consider the composition $$ P_{2k}(\emptyset) \otimes V_t(\lambda) \twoheadrightarrow P_{2k}(\emptyset) \otimes L_t(\lambda)  \twoheadrightarrow Y \twoheadrightarrow P_k(\lambda)$$ where the first map is induced by the epimorphism $V_t(\lambda) \twoheadrightarrow L_t(\lambda)$ (cf. Subsection \ref{ssec:properties_cat_V_t}).
Applying Steps 1 and 2, we conclude that there exists an epimorphism $T \twoheadrightarrow P_k(\lambda) $ where $T$ is in the image of $\mathcal{D}_t$. 

{\bf Step 4}:
Let $k$ be such that $M$ belongs to $\cV^k_t$. The category $\cV^k_t$ has enough projectives, so there exists an epimorphism $ P \twoheadrightarrow M$ where $P$ is a projective object in $\cV^k_t$. Applying Step 3, we obtain an epimorphism $T \twoheadrightarrow P$, with $T$ in the image of $\mathcal{D}_t$; composed with the former, it gives an epimorphism $$ T \twoheadrightarrow M$$ as wanted.
\end{proof}

We begin with the proof of Step 1 of Proposition \ref{prop:presentation}.
\begin{lemma}\label{lem:trivialcosocle} Let $t=m-n$,  $\g={\mathfrak{gl}(m|n)}$ and $P(0)\in \rp(\g)$ be the projective cover of the trivial module.
Then there exists\footnote{ We emphasize that the object $Y$ constructed here depends on the integers $m, n$.}  $Y\in \cD_t$ such that $F_{m,n}\circ I(Y)=P(0)$ and the cosocle of $I(Y)$ in $\cV_t$ is isomorphic to the unit object $\triv \in \cV_t$. Moreover, $$F_{m, n}(I(Y) \twoheadrightarrow \triv): P(0) \rightarrow \triv \in \rp(\g)$$ is a epimorphism.
\end{lemma}

\begin{proof} 
%
We will prove the statement for $t\geq 0$. The case of negative $t$ is similar.
Consider the weight diagram $f_0$ of the zero weight. It is
$$\dots\circ\times\dots\times>\dots>\circ\dots, $$
with $n$ symbols $\times$ and $t$ symbols $>$. Let $\lambda$ be a weight with weight diagram $f_\lambda$
$$\dots\circ<\dots<>\dots>\circ\dots,$$
with $n$ symbols $<$ and $m$ symbols $>$. In the language of bipartitions we have $\lambda^\circ= (n, n, ..., n)$ ($m$ times) and $\lambda^{\bullet}= \emptyset$.

One can easily check that $f_0$ can be obtained from $f_{\lambda}$ by a sequence of the following elementary moves
\begin{itemize}
\item moving $>$ to the adjacent left empty position;
\item changing $<>$ to $\times\circ$;
\item changing $<\times$ to $\times<$.
\end{itemize}

By Lemma \ref{lem:GS} this implies that there is a sequence of translation functors $T_1,\dots, T_r$ ($r=mn$) such that
$P(0)=T^*_r\circ\dots\circ T^*_1(P(\lambda))$. Here $T_i = T_{\theta, \chi}$ for some cores 
$\theta, \chi$ such that $ u(\chi) \geq u(\theta)$, and $T^*_i = T^*_{\chi, \theta}$ is its adjoint (on either side).


Note that $\lambda$ is a typical weight in $\rp(\g)$, hence $P(\lambda)=L(\lambda)$.

Now we set
$R:=\overline{T}^*_r\circ\dots\circ\overline{T}^*_1(L_t(\lambda))$.

By definition, $F_{m,n}(L_t(\lambda))=L(\lambda)$, so Lemma \ref{lem:infblocks} implies:
$P(0)=F_{m,n}(R)$. 

On the other hand, $L_t(\lambda)$ is a subquotient in $V_t^{ \otimes |\lambda|} = V_t^{ \otimes mn}$. Recall from \cite{D}, \cite{CW} that the object 
$X_t^{ \otimes mn}$ of $\cD_t$ is a direct sum of indecomposables $S^{\mu} X_t$ which satisfy $$\dim \Hom_{\cD_t}(S^{\mu} X_t, S^{\tau} X_t) = \delta_{\mu, \tau}$$ 
Therefore the object $V_t^{ \otimes mn}$ in $\cV_t$ is semisimple, and $L_t(\lambda)$ is a direct summand of $V_t^{ \otimes mn}$.

From the definition of $R$ it follows that $R$ is a direct summand in the mixed tensor power $V_t^{\otimes mn}\otimes (V_t^*)^{\otimes mn}$. 
Hence $R=I(Y)$ for some $Y\in \cD_t$.

\mbox{}

Next we prove that the cosocle of $R$ is simple. 

Denote $R_i:=\overline{T}^*_{i}\circ\dots\circ\overline{T}^*_1(L(\lambda))$. We prove by induction on $i$ that $\operatorname{cosoc}(R_i)$ is simple. 

Assume that cosocle of $R_i$ is simple. By the biadjointness of the pair of functors $(\overline{T}_i, \overline{T}_i^*)$, we have
$$\operatorname{Hom}_{\cV_t}(\overline{T}^*_{i+1}(R_i),L)=\operatorname{Hom}_{\cV_t}(R_i,\overline{T}_{i+1}(L)).$$
for any simple $L$ in the corresponding block of $\cV_t$. 

 By Remark \ref{rem:inftranslation} we can apply Lemma \ref{lem:GS}(e) for translation functors $\overline{T}_{i+1}$ in
$\cV_t$. Hence $\overline{T}_{i+1}(L)$ is either simple or zero, and 
$L \subset \operatorname{cosoc}(R_{i+1})$ iff $\overline{T}_{i+1}(L) \subset \operatorname{cosoc}(R_i)$.

By the induction hypothesis, $\operatorname{cosoc}(R_i)$ is simple so it remains to check that there exists at most one isomorphism class of simple 
objects $L$ such that $\overline{T}_{i+1}(L)=\operatorname{cosoc}(R_i)$. This follows from Lemma \ref{lem:GS}(e).

Hence $\operatorname{cosoc}(R_{i+1})$ is simple, and $\overline{T}_{i+1}(\operatorname{cosoc}(R_{i+1}))=\operatorname{cosoc}(R_{i})$. Thus
$$L_t(\lambda)=\overline{T}_1\circ\dots\circ\overline{T}_r(\operatorname{cosoc}(R))$$ 

We now prove that $\operatorname{cosoc}(R) = \triv \in \cV_t$. Applying Lemma \ref{lem:GS}(e) again, we see that it is equivalent to proving that $$L_t(\lambda)=\overline{T}_1\circ\dots\circ\overline{T}_r(\triv).$$ Both sides are simple objects in $\cV_t$ (see Lemma \ref{lem:GS}(e)), and the isomorphism holds since their images under the functor $F_{m, n}$ are not zero and coincide.

\mbox{}

It remains to prove that $F_{m, n}(R \rightarrow \triv): P(0) \rightarrow \triv \in \rp(\g)$ is a epimorphism. Denote $$\overline{T} :=\overline{T}_1\circ\dots\circ\overline{T}_r.$$ Notice that in $\cV_t$, the epimorphism $R \twoheadrightarrow \triv$ can be defined as $$\varepsilon^{\overline{T}^*, \overline{T}} \rvert_{\triv}: R=\overline{T}^* \circ \overline{T} (\triv) \longrightarrow \triv$$ where $\varepsilon^{\overline{T}^*, \overline{T}}$ is the counit for the adjunction $(\overline{T}^*, \overline{T})$. By Lemma \ref{lem:infblocks}, $F_{m, n}(\varepsilon^{\overline{T}^*, \overline{T}}) = \varepsilon^{T^*, T}$, where $$T :=T_1\circ\dots\circ T_r$$ in $\rp(\g)$ and $\varepsilon^{T^*, T}$ is the counit for the adjunction $(T^*, T)$. In particular, $$F_{m, n}(\varepsilon^{\overline{T}^*, \overline{T}} \rvert_{\triv}:R \rightarrow \triv) = \varepsilon^{T^*, T} \rvert_{\triv} $$ which is an epimorphism.
\end{proof}

\begin{lemma}\label{lem:trivialprojective} Let $P_k(\emptyset)$ denote the projective cover of $\triv$ in $\cV^k_t$. Then $P_k(\emptyset)$ is a quotient of 
some object in $I(\cD_t)$.
\end{lemma}

\begin{proof}
Let $m,n$ be such that $\cV^k_t$ is equivalent to $\rp^k(\g)$, $\g=\mathfrak{gl}(m|n)$. Let $P(0)$ denote the projective cover of the trivial module
in $\rp(\g)$ and we set 
$$Q:= j^k_!(P(0)) = F_{m, n}(P_k(\emptyset))$$
where $j^k_!$ is the left adjoint to the inclusion functor $j^k: \rp^k(\g) \rightarrow \rp(\g)$ (cf. Subsection \ref{ssec:properties_cat_V_t}).

In particular, $Q$ is the projective cover of the trivial module in $\rp^k(\g)$.

Next we consider $R:=I(Y)$, where $Y$ is the object obtained in Lemma \ref{lem:trivialcosocle} for the pair $(m, n)$.

Let $$Z := \mathit{i}^k_!(R)$$ (notation as in Subsection \ref{ssec:properties_cat_V_t}). Then $Z$ is the maximal quotient of $R$ which lies in $\cV_t^k$, and we denote by $\phi$ the epimorphism $R \twoheadrightarrow Z$:
$$ \xymatrix{ &R \ar@{->>}[r]^{\phi} \ar@{->>}[rd] &Z \ar@{->>}[d]^{\pi} \\ &{} &\triv}
$$
We also note that $cosoc(Z) \cong \triv$.

We now apply $F_{m, n}$ to the above diagram.


The map $F_{m, n}(\phi): F_{m, n}(R) \cong P(0) \longrightarrow F_{m, n}(Z)$ factors uniquely through $p:P(0) \twoheadrightarrow Q$ and we obtain a map $p': Q \rightarrow F_{m,n}(Z)$ such that $$ F_{m, n}(\phi) = p'\circ p$$ 

We shall prove that $p'$ is an isomorphism. We start by showing that it is an epimorphism.

\mbox{}

Consider the morphism $$F_{m, n}(R \twoheadrightarrow \triv): P(0) \rightarrow \triv$$

It is equal to $F_{m, n}(\pi) \circ F_{m, n}(\phi) = F_{m, n}(\pi) \circ p' \circ p$.

By Lemma \ref{lem:trivialcosocle}, $F_{m, n}(R \twoheadrightarrow \triv)$ is an epimorphism. Therefore $F_{m, n}(\pi) \circ p':Q \twoheadrightarrow \triv$ is an epimorphism. We obtain the following commutative diagram
$$ \xymatrix{ &P(0) \ar@{->>}[rr]^{F_{m, n}(\phi)} \ar@{->>}[rrrd] \ar@{->>}[d]_{p} &{} &F_{m, n}(Z) \ar@{->>}[rd]^{F_{m, n}(\pi)} &{} \\ &Q \ar@{->>}[rrr] \ar[rru]|{p'} &{} &{} &\triv}
$$

We now recall that $cosoc(F_{m, n}(Z)) \cong \triv$ (since $F_{m, n}: \cV^k_t \rightarrow \rp^k(\g)$ is an equivalence). This implies that $p'$ is an epimorphism. 

\mbox{}

Next, we show that $p'$ is a monomorphism.

Recall that by definition of $\rp^k(\gl(m|n))$, $Q$ is a subquotient (and thus a subobject, since $Q$ is projective) of some finite direct sum of $T^{p, q}$, 
$p + q \leq k$. This implies that there exists an inclusion $f:  Q \hookrightarrow F_{m, n}(D)$, where $D$ is an object in $ I (\mathcal{D}_t^k) $.

Next, the functor $F_{m, n} \circ I : \cD_t \rightarrow \rp({\gl(m|n)})$ is full (\cite{BS}, \cite{CW}), so there exists 
$\alpha: R \rightarrow D$ such that $F_{m, n}(\alpha) = f \circ p$. 

By definition of $Z$, $\alpha$ factors through $\phi$ and we obtain $\alpha': Z \rightarrow D$ such that $\alpha' \circ \phi = \alpha$.

By definition of 
$\alpha$, $$f \circ p = F_{m, n}(\alpha) = F_{m, n}(\alpha') \circ F_{m, n}(\phi) = F_{m, n}(\alpha') \circ p' \circ p$$

Since $p: P(0) \rightarrow Q$ is surjective, by cancellation law we have: $$f = F_{m, n}(\alpha') \circ p'$$ 
In particular, we conclude that $p': Q \to F_{m, n}(Z)$ is a monomorphism, since $f = F_{m, n}(\alpha') \circ p'$ is.

The following commutative diagrams sum up the above constuctions:

In $\cV_t$, we have
$$ \xymatrix{ &R \ar[r]^{\alpha} \ar@{->>}[d]_{\phi}  &D \\   &Z \ar[ur]_{\alpha'} \ar@{->>}[r]_{\pi} &\triv}$$

In $\rp(\gl(m|n))$, we have
$$ \xymatrix{ &{P(0) = F_{m, n}(R)} \ar[r]^-{f \circ p} \ar@{->>}[d]_{p} \ar[rrrd]|{F_{m, n}(\phi)} &F_{m, n}(D)  &{} &{} &{} &{} \\ &Q 
\ar@{->>}[rrr]_{p'} \ar@{^{(}->}[ru]|f &{} &{} &F_{m, n}(Z) \ar@/_1.5pc/[llu]_{F_{m, n}(\alpha')} \ar@{->>}[rr]^{F_{m, n}(\pi)} &{} &\triv} $$

%
%
%

Thus $p'$ is an isomorphism and $Z$ is isomorphic to $ P_k(\emptyset)$ (since $\rp^k(\g)$ is equivalent to $\cV^k_t$), the projective cover of $\triv$ in $ \cV^k_t$, and is a quotient of $R = I(Y)$, $Y \in \cD_t$.

\end{proof}

We now prove Step 2 of Proposition \ref{prop:presentation}.
\begin{lemma}\label{lem:simplequotient} Let $V_t(\lambda)$ be a standard object in $\cV_t$. There exists an object $D\in I(\cD_t)$ such that $V_t(\lambda)$ is a
quotient of $D$.
\end{lemma}
\begin{proof} Recall the category $\T_\tg$ and the exact SM functor $\Phi:\T_\tg\to\cV_t$. This functor takes simple objects $\tV(\lambda)$ in $\T_\tg$ to standard objects $V(\lambda)$; furthermore, for any injective $E \in \T_\tg$, $\Phi(E) \in I(\cD_t)$. Thus it is enough to show that for any simple object $\tV(\lambda)$ in $\T_\tg$ there exists an injective object $E$ and an epimorphism $E \twoheadrightarrow \tV(\lambda)$. We will show that there exists a bipartition $\nu$ such that 
the injective object $$\tilde Y(\nu) := \tV(\nu^{\circ}) \otimes \tV(\nu^{\bullet})$$
has a quotient isomorphic to $\tV(\lambda)$.

We recall from \cite[Theorem 2.3]{PS} the following multiplicity formula for $r \geq 1$:
\begin{equation}\label{eq:mult}
[\overline{soc}^{r}(\tilde Y(\nu)):\tV(\lambda)]=\sum_{\gamma: |\gamma|=r-1} N^{\nu^\circ}_{\lambda^\circ,\gamma}N^{\nu^\bullet}_{\lambda^\bullet,\gamma},
\end{equation}
where $\overline{soc}^{r}:=soc^{r}/soc^{r-1}$ denotes the $r$-th Loewy layer in socle filtration of $\tilde Y(\nu)$, $N^\alpha_{\beta,\gamma}$ denote Littlewood--Richardson coefficients and $\gamma$ is a partition of size $r-1$. 

Let $\delta$ be a rectangular partititon with height and width greater than $|\lambda|$. Let $\nu^\circ$ be obtained by adding $\lambda^\circ$ to the right of $\delta$ and $\nu^\bullet$ by adding
$\lambda^\bullet$ to the bottom of $\delta$. 

We claim that by \eqref{eq:mult},
\begin{align}\label{eq:mult2}
 [\tilde Y(\nu):\tV(\lambda)] = 1 \;\; \;\; \;\; \text{ and } \;\; \;\; \;\; [\tilde Y(\nu):\tV(\mu)]=0
\end{align}

for any bipartition $\mu$ with $|\mu|<|\lambda|$. If we prove this, then \eqref{eq:mult} would imply that $\tV(\lambda)$ lies in the cosocle of $\tV(\lambda)$, and we are done. 

The equalities \eqref{eq:mult2} follow from the following facts: 
\begin{itemize}
\item If $N^{\nu^\circ}_{\mu^\circ,\gamma} \neq 0$ then the number of rows of $\gamma$ is less or equal the number of rows of $\delta$.
 \item If $N^{\nu^\bullet}_{\mu^\bullet,\gamma}\neq 0$ then the number of columns of $\gamma$ is less or equal the number of columns of $\delta$.
\end{itemize}
 The first statement can be easily obtained from the combinatorial description of the Littlewood-Richardson coefficients (see \cite{FH})
\footnote{Alternatively, one can show that for any bipartition $\mu$ such that $|\mu|\leq |\lambda|$, $N^{\nu^\circ}_{\mu^\circ,\gamma}\neq 0$ implies that 
$\ell(\gamma) \leq \ell(\nu^\circ)=\ell(\delta)$ (here $\ell$ denotes the number of rows in a partition). This is a straightforward consequence of the definition of 
Littlewood-Richardson coefficients.}. The second statement follows from the first by the transpose symmetry of the Littlewood--Richardson coefficients.

%
%
%
%

The above statements imply that if
$N^{\nu^\circ}_{\mu^\circ,\gamma}N^{\nu^\bullet}_{\mu^\bullet,\gamma}\neq 0$, then $\gamma$ is a Young subdiagram of the rectangle diagram $\delta$. On the other hand, since the above Littlewood--Richardson coefficients are non-zero, we have $$|\gamma|= |\nu^\circ|-|\mu^\circ| = |\nu^\bullet|-|\mu^\bullet| $$ and so
$$ |\gamma| = \frac{1}{2}\left( |\nu|-|\mu|  \right) \geq \frac{1}{2}\left( |\nu|-|\lambda|  \right)  = |\delta|$$
Hence $[\tilde Y(\nu):\tV(\mu)]=N^{\nu^\circ}_{\mu^\circ,\gamma}N^{\nu^\bullet}_{\mu^\bullet,\gamma}\neq 0$ iff $\gamma = \delta$, which is possible only when $|\mu| = |\lambda|$.
Moreover, $N^{\nu^\circ}_{\lambda^\circ,\delta}=N^{\nu^\bullet}_{\lambda^\bullet,\delta}=1$, hence $[\tilde Y(\nu):\tV(\lambda)]=1$.

\end{proof}

\begin{remark}
 It is worth mentioning that Proposition \ref{prop:presentation} implies that the functors $$F_{m, n}: \cV_{t=m-n} \longrightarrow \rp(\g)$$ are full. This follows from the fact that the functors $F_{m, n} \circ I: \cD_{t=m-n} \longrightarrow \rp(\g)$ are full (cf. \cite{CW}).
\end{remark}

\subsection{Highest weight structure}
\begin{lemma}\label{lem:filtration} An indecomposable projective object $P_k(\lambda)$ has filtration by the standard objects $V_t(\mu)$ with
$\mu\geq\lambda$.
\end{lemma}
\begin{proof} By abuse of notation, we will denote $X_t^{p, q} := V_t^{\otimes p } \otimes {V_t^*}^{\otimes q}$ throughout this proof.

We have proved in the previous subsection that there exists an object $T\in I(\cD_t)$ such that $P_k(\lambda)$ is the maximal quotient of $T$ which 
belongs to $\cV_t^k$. We also know that $T$ is a direct summand in $X_t^{p, q}$ for some $p,q$.
To describe the kernel $N\subset T$ of the surjection $T\twoheadrightarrow P_k(\lambda)$ 
choose a basis $\psi_i$ in $$\bigoplus_{\min(p,q)\geq  s \geq \frac{p+q-k}{2}}\Hom(X_t^{p, q},X_t^{p-s, q-s}).$$ We can choose this basis from the diagram bases for these walled Brauer algebras (see \cite{BS} or \cite{CW}). First, we show that
$$N=T\cap\bigcap_{i} \Ker \psi_i$$ The inclusion $\subset$ is obvious, so we only need to prove $\supset$. 

Denote $K:= \bigcap_{i} \Ker \psi_i$. With the notation as in Subsection \ref{ssec:properties_cat_V_t}, we have: $$K=\bigcap_{i} \Ker \psi_i = \Ker (X_t^{p, q}) \rightarrow \mathit{i}^k_!(X_t^{p, q})$$ Thus the composition $$K \subset X_t^{p, q} \twoheadrightarrow T \twoheadrightarrow \mathit{i}^k_!(T) = P_k(\lambda)$$ is zero, which implies $T \cap K \subset N$.

Next, fix $r=\lceil \frac{p+q-k}{2} \rceil$.
Note that any diagram $d$ with $p,q$ nodes in the top row and $p-s,q-s$ nodes in the bottom row ($\min(p,q)\geq  s \geq r$) has at least $r$ horizontal links in the top row. So $d$ can be written as a product of diagrams $d = d_2 \circ d_1$ such that $d_1$ has no horizontal links in the bottom row and has $p,q$ nodes in the top row and $p-r,q-r$ nodes in the bottom row.


If we recall $\Phi:\mathbb T_{\tilde\g}\to \cV_t$, this implies that any $\psi_i$ factors as $$\psi_i = \varphi'_i \circ \varphi_i$$ where $\varphi_i\in \Phi\Hom(\tT^{p,q},\tT^{p-r,q-r})$, and $\varphi'_i \in\Hom(X_t^{p-r, q-r}, X_t^{p-s, q-s})$. 

Moreover, for any $\varphi\in \Phi\Hom(\tT^{p,q},\tT^{p-r,q-r})$ represented by a diagram of type $d_1$ one can find $i$ such that $\psi_i = \varphi$. 

Thus we have 
$$\bigcap_{i} \Ker \psi_i \cong \bigcap_{\varphi\in \Phi\Hom(\tT^{p,q},\tT^{p-r,q-r})}\Ker \varphi.$$  

Let $$\tilde K := \bigcap_{\phi\in \Hom(\tT^{p,q},\tT^{p-r,q-r})}\Ker \phi$$ be the corresponding object in $\mathbb T_{\tilde\g}$. 
This is the $r+1$-th term in the socle filtration of $\tilde T^{p,q}$ (see \cite{PS}).

Recall that the functor $\Phi$ is exact (see Lemma \ref{lem:functortg}), so it preserves finite limits and thus $$\Phi(\tilde{K}) \cong K$$

Let $X^{p,q}=T_1\oplus\dots\oplus T_s$ be a direct sum of indecomposable objects with $T_1=T$. Set
$$F^i(X^{p,q}):=\bigcap_{\varphi\in\Hom
  (X^{p,q},X^{p-i,q-i})}\Ker\varphi,\quad
F^i(T_j):=\bigcap_{\varphi\in\Hom (T_j,X^{p-i,q-i})}\Ker\varphi.$$
In particular, $F^r( X^{p,q}) = K$, $F^r(T) = N$.

If we denote by $e_j$ the projector $X^{p,q}\to T_j$, then for any $\varphi\in\Hom (X^{p,q},X^{p-i,q-i})$ we have
$$\bigcap_{j=1}^s\Ker(\varphi\circ e_j)=\bigoplus_{j=1}^s\Ker\varphi\cap T_j\subset\Ker\varphi.$$
That implies
$$F^i(T_j)=F^i(X^{p,q})\cap T_j,\quad F^i(X^{p,q})=\bigoplus_jF^i(T_j),$$
since
$$\bigcap_{\varphi\in\Hom(X^{p,q},X^{p-i,q-i})}\Ker\varphi=\bigcap_{\varphi\in\Hom(X^{p,q},X^{p-i,q-i})}\bigcap_j\Ker(\varphi\circ e_j).$$

On the other hand, $F^i(X^{p,q})/F^{i-1}(X^{p,q})$ is a direct sum of standard objects $V_t(\mu)$ for some $\mu$, hence
by Krull-Schmidt theorem $F^i(T)/F^{i-1}(T)$ is a direct sum of
standard objects. As it was shown above, $P_k(\lambda)\simeq T/F^r(T)$. Hence $P_k(\lambda)$ has a filtration 
by standard objects.

It remains to prove that all $V_t(\mu)$ which occur in $P_k(\lambda)$
satisfy the condition $\mu\geq\lambda$. For this we use Proposition
\ref{prop:kacrep}, which claims
$$\dim\Hom(V_t(\mu),\check V_t(\nu))=\delta_{\mu,\nu},\quad \Ext^1(V_t(\mu),\check V_t(\nu))=0.$$
If $V_t(\mu)$ occurs in $P_k(\lambda)$, then $\Hom(P_k(\lambda),\check V_t(\mu))\neq 0$, hence 
$[\check V_t(\mu):L_t(\lambda)]\neq 0$. The latter implies $\mu\geq \lambda$.
\end{proof}
\begin{corollary}\label{cor:highestweight} For any $k\geq 0$ the category $\cV_t^k$ is a highest weight category with duality, in the sense of \cite{CPS}, with standard objects (up to isomorphism) $V_t(\lambda)$, $|\lambda| \leq k$.
Hence the inductive completion of $\cV_t$ is also a highest weight category (with infinitely many weights).
\end{corollary}

\begin{remark}
 The objects $X \in I(\mathcal{D}_t)$ which lie in $\cV_t^k$ are both standardly filtered and self dual, making them tilting objects in the highest weight category $\cV_t^k$. This is similar to the situation in the abelian envelope of the Deligne category $\rp(S_t)$, see \cite{CO}.
\end{remark}

\subsection{Epimorphisms \texorpdfstring{in $\cV_t$}{}}

Let us calculate the spaces $\Hom_{\mathcal{D}_t}(X_t^{\otimes r} \otimes X_t^{* \otimes r}, \triv)$.

Recall that Schur functors $S^\lambda X_t$ are indecomposable in $\cD_t$
and 
\begin{equation}
\Hom_{\cD_t}(S^\lambda X_t,S^\mu X_t)=\begin{cases}0, & \lambda\ne\mu,\\
\C\cdot\id, & \lambda=\mu.\end{cases}
\end{equation}
(see \cite{D, CW}).

Schur-Weyl decomposition 
$$X_t^{\otimes r} = \bigoplus_{\lambda \vdash r}Y_{\lambda}\otimes S^{\lambda} X_t,$$ 
with $\lambda$ running over the set of all partitions of $r$, and $Y_{\lambda}$ 
being an irreducible $S_r$-representation corresponding to $\lambda$, gives us a decomposition
\begin{equation}
\label{eq:decomposition}
\Hom_{\mathcal{D}_t}(X_t^{\otimes r} \otimes X_t^{* \otimes r}, \triv) \cong \bigoplus_{\lambda \vdash r} \left( Y_{\lambda} \otimes Y^*_{\lambda} \right)
\cdot  ev_{S^{\lambda} X^*_t},
\end{equation}
where the evaluation map $ev_{S^{\lambda} X^*_t}$ is a generator of the one-dimensional
space $\Hom_{\mathcal{D}_t}(S^\lambda X_t \otimes S^\lambda X_t^*, \triv)$.

Epimorphisms in $\cV_t$ satisfy a very agreeable property.

\begin{proposition}\label{prp:splitting-ses}
  For any epimorphism $ M \stackrel{g}{\longrightarrow} M'$ in $\cV_t$ there exists a nonzero object $Z$ in $\cD_t$ such that the epimorphism
  $ M \otimes I(Z) \stackrel{g \otimes \id }{\longrightarrow} M' \otimes I(Z)$ is split. A similar property is valid for monomorphisms.
 \end{proposition}

\begin{proof}
 
 We start by considering the case $M'=\triv$. We can choose an epimorphism 
$I(X)\twoheadrightarrow M$ with $X\in\cD_t$. We obtain an epimorphism $\bar{g}:I(X) \rightarrow \triv$. 
 
 It would then be enough to find a nonzero object $Z\in\cD_t$   so that the epimorphism $\bar{g} \otimes \id_Z: T(X) \otimes T(Z) \rightarrow T(Z)$ is split.
 
Since any object of $\cD_t$ is a retract of the mixed tensor power of the generator
$X_t$, 
it is enough to verify the above statement in the case when $I(X) = V_t^{\otimes p} \otimes V_t^{* \otimes q}$. Moreover, since $\bar{g}$ is a epimorphism , we need to have $p=q$.
 
 Formula \eqref{eq:decomposition} tells us that there exists an embedding
$$ i:S^\lambda V_t\otimes S^\lambda V_t^* \to V_t^{\otimes p}\otimes V_t^{*\otimes  p}$$
such that the composition $\bar g\circ i$ is a nonzero multiple of 
$ev_{S^\lambda V_t}$. Therefore,  it is enough to verify that given a Young diagram $\lambda$, there exists an nonzero object $Z\in\cD_t$,  such that the epimorphism 
 $$ \begin{CD}                                                                                                                                                                                                                                                                                                      I(Z) \otimes S^{\lambda} V_t \otimes S^{\lambda} V_t^* @>{\id_Z \otimes ev_{S^{\lambda} V_t^*}}>> Z                                                                                                                                                                                                                                                                                                                    \end{CD}$$
 is split. But this is obviously true, for instance, for  $Z:= S^{\lambda} X_t^*$, when the statement  follows from the definition of a dual object.

 \mbox{}
 
 In the general case, we have an epimorphism $g:M \twoheadrightarrow M'$. 
Consider the pullback of the epimorphism 
$$g \otimes \id_{M'^*}:M\otimes M'^* \twoheadrightarrow M'\otimes M'^*$$ 
along $coev: \triv \rightarrow M' \otimes M'^*$. 
We obtain an epimorphism 
$$ (M\otimes M'^*) \times_{M'\otimes M'^*} \triv \twoheadrightarrow \triv.$$ 
We already know that there exists a nonzero $Z\in\cD_t$ such that the epimorphism $$I(Z) \otimes \left( (M\otimes M'^*) \times_{M'\otimes M'^*} \triv\right)   \twoheadrightarrow I(Z)$$ 
splits.

This gives us a morphism $\phi:I(Z) \rightarrow I(Z) \otimes (M\otimes M'^*)$ such that the following diagram is commutative:
 $$ \xymatrix{ &{I(Z) \otimes M\otimes M'^* } \ar[rr]^{\id_{I(Z)} \otimes g \otimes \id_{M'^*}} &{} &{I(Z) \otimes M'\otimes M'^*} \\ &{} &{} &I(Z) \ar[u]_{\id_Z \otimes coev } \ar[llu]^{\phi}}.$$
Denote the image of $\phi$ under the isomorphism $$\Hom_{\cV_t}(I(Z), I(Z) \otimes M\otimes M'^*) \stackrel{\sim}{\longrightarrow} \Hom_{\cV_t}(I(Z)\otimes M'^*, 
I(Z) \otimes M)$$ by $\bar{\phi}$. The above commutative diagram implies 
that $(\id_{I(Z)} \otimes g) \circ \bar{\phi} = \id_{I(Z)}\otimes M'^*$, 
i.e. the epimorphism $\id_{I(Z)}\otimes g$ splits.
\end{proof}

\section{Universal property}\label{sec:universal_prop}


In this section we will prove a universal property of the functor $I:\cD_t\to\cV_t$.
We use \InnaB{only} a few facts about the categories $\cD_t$ and $\cV_t$. 
This is why we believe the universality theorem we prove will be applicable to
other contexts. Therefore, we present it as a general theorem for a pair of SM categories satisfying certain properties we now list.
\begin{remark}
 \InnaB{The results below are stated for $\C$-linear categories, but hold for $\mathbbm{k}$-linear categories, where $\mathbbm{k}$ is any field.}
\end{remark}

\subsection{Assumptions}
\label{ss:ass}

In what follows $I:\cD\to\cV$ is a symmetric monoidal functor from an additive
$\C$-linear rigid SM category $\cD$ to a tensor $\C$-linear category $\cV$.

We assume the following.
\begin{enumerate}\label{req_subcat_D}
\item\label{req_subcat_D1} $I:\cD\to\cV$ is fully faithful.
\item\label{req_subcat_D2} Any $X\in\cV$ can be presented as an image of a 
map $I(f)$
for some $f:P\to Q$ in $\cD$.
\item\label{req_subcat_D3} For any epimorphism $X\to Y$ in $\cV$ there 
exists a nonzero $T\in\cD$
such that the epimorphism $X\otimes I(T)\to Y\otimes I(T)$ splits.
\end{enumerate}

Our functor $I:\cD_t\to \cV_t$ satisfies the above properties by Proposition \ref{prop:deligneconnection}, Proposition \ref{prop:presentation} and 
Proposition \ref{prp:splitting-ses}.

Our main theorem asserts that if a functor $I:\cD\to\cV$ satisfies the above properties, it is universal in the sense we will now formulate.

\

Universality of the functor $I:\cD\to\cV$ is naturally 2-categorical. We will
note that 2-category (in this paper) $\fZ$ means a category enriched over
categories. Thus, $\fZ$ has objects, and a category of morphisms $\Map_\fZ(x,y)$ between each pair of objects $x,y\in\fZ$ with a strictly associative composition.
  
$\fZ$ is enriched over groupoids if all $\Map_\fZ(x,y)$ are groupoids.
A groupoid is called contractible if there is a unique arrow (it is automatically an isomorphism) between any two objects.

\subsection{Main universality result}

We define a 2-category $\fX$, more precisely, a category enriched over groupoids,
as follows.

The objects of $\fX$ are pairs $(F,\cA)$ where $\cA$ is a \InnaA{$\C$-linear} 
tensor
category and $F:\cD\to\cA$ is a faithful symmetric monoidal (SM) 
\InnaA{$\C$-linear} functor.
 
Given two objects $(F,\cA)$ and $(G,\cB)$ in $\fX$, the groupoid $\Map_\fX(F,G)$
is defined as follows.

\begin{itemize}
\item 
Its objects are pairs $(U,\theta)$ where $U:\cA\to\cB$ is an exact SM 
\InnaA{$\C$-linear} functor 
and $\theta:U\circ F\to G$ is a SM isomorphism of functors.
\item
A morphism from $(U,\theta)$ to $(U',\theta')$ is a SM isomorphism of functors 
$U\to U'$ commuting with $\theta$ and $\theta'$. 
\end{itemize}

We are now able to formulate our main universality result.

\begin{theorem}\label{Mthm:uni-prop}
Assume that the functor $I: \cD \to\cV$ satisfies the assumptions 
\eqref{req_subcat_D1}--\eqref{req_subcat_D3}.
Then $I$ is an initial object in $\fX$.
The latter means that for any object $F:\cD \to \cA$ the groupoid
$\Map_\fX(I,F)$ is contractible. 
\end{theorem}


The following result can be easily shown to be equivalent to
Theorem~\ref{Mthm:uni-prop}. 
\begin{theorem}\label{Mthm:uni-2}
Under the same assumptions the functor $I$ induces for any \InnaC{tensor} category $\cA$ an equivalence of the following categories
\begin{itemize}
\item $\Fun^\ex(\cV,\cA)$, the category of exact SM \InnaA{$\C$-linear} 
functors $\cV\to\cA$,
\item $\Fun^\faith(\cD,\cA)$, the category of faithful SM 
\InnaA{$\C$-linear} functors $\cD\to\cA$. 
\end{itemize} 
\end{theorem}
\begin{remark}
 \InnaC{One can weaken the condition that $\cA$ is a tensor category by not requiring rigidity: namely, $\cA$ is an abelian $\C$-linear SM category with biexact, bilinear bifunctor $-\otimes -$ and a simple unit object $\triv \in \cA$ (hence $\End_{\cA}(\triv)=k$). In that case any dualizable object in $\cA$ is faithfully flat on the full subcategory of dualizable objects in $\cA$.}
\end{remark}

\begin{proof}

The categories of functors in question are groupoids (see for example \cite[Section 2.7]{D1}). A functor 
$f:G\to H$ between groupoids is an equivalence if and only if
for each $h\in H$ the fiber 
$$f_{/h}=\{(g,\theta)|g\in G,\theta:f(g) \stackrel{\sim}{\longrightarrow} h\}$$
is contractible (in particular, non-empty).

The composition with the functor $I:\cD\longrightarrow\cV$ yields the functor
$\Fun^\ex(\cV,\cA)\longrightarrow\Fun^\faith(\cD,\cA)$. Its fiber over $F:\cD\to\cA$
is precisely $\Map_\fX(I,F)$.

This proves the theorem.
\end{proof}

\subsubsection{}
Under the same assumptions \eqref{req_subcat_D1}--\eqref{req_subcat_D3} the 
functor $I:\cD\to\cV$ satisfies another universal property which we will now 
formulate. It has nothing to do with the 
SM structure of the categories involved and it would not be very appealing, 
would it  not appear as an intermediate step in the proof of Theorem~\ref{Mthm:uni-prop}.

\subsection{Pre-exact functors}
 
Let $\cC$ be an additive category endowed with two collections of arrows:
{\sl inflations} (playing the role of monomorphisms) and {\sl deflations} (playing the role of epimorphisms). The only important example for us is the category $\cD$ with inflations defined as the arrows becoming monomorphisms in $\cV$, and deflations becoming epimorphisms in $\cV$.
 
An additive functor $\cC\to\cB$ to an abelian category $\cB$ is called 
{\it pre-exact} if it takes inflations to monomorphisms and deflations to epimorphisms.
We will also call an additive functor $\cC\to\cB$ between two abelian categories pre-exact, if it preserves monomorphisms and epimorphisms.
 
\subsubsection{}
We define a 2-category $\fY$ as follows.

Its objects are pairs $(F,\cA)$ where $\cA$ is a $\C$-linear abelian
category and $F:\cD\to\cA$ is a \InnaA{$\C$-linear faithful} pre-exact functor.
 
Given two objects $(F,\cA)$ and $(G,\cB)$ in $\fY$, the groupoid $\Map_\fY(F,G)$
is defined as follows.

\begin{itemize}
\item 
Its objects are pairs $(U,\theta)$ where $U:\cA\to\cB$ is a 
\InnaA{$\C$-linear} pre-exact functor 
(that is, the one preserving monomorphisms and epimorphisms) and $\theta:U\circ F\to G$ is an isomorphism of functors.
\item
A morphism from $(U,\theta)$ to $(U',\theta')$ is an isomorphism of functors 
$U\to U'$
commuting with $\theta$ and $\theta'$. 
\end{itemize}

We claim the following 
\begin{theorem}\label{Mthm:uni-prop-0}
Under the assumptions \eqref{req_subcat_D1}--\eqref{req_subcat_D3}, the functor 
$I: \cD \to\cV$ is an initial 
object in $\fY$;
that is, for any object $F:\cD \to \cA$ the groupoid
$\Map_\fY(I,F)$ is contractible. 
\end{theorem}

\subsection{\texorpdfstring{}{Connection between the universality results}}

The assumptions \eqref{req_subcat_D1}--\eqref{req_subcat_D3} are assumed 
throughout this and the next section.

We will prove in Lemma~\ref{Mlem:preexact} below that a faithful SM functor $F:\cD\to\cA$ to a tensor category is necessarily pre-exact. This implies that an obvious forgetful functor $\#:\fX\to\fY$ is defined. The following result justifies our interest to  the 2-category $\fY$.

\begin{proposition}\label{Mprp:XtoY}
The forgetful functor $\fX\to\fY$ induces for any $F:\cD\to \cA$ in $\fX$ an equivalence
\begin{equation}
\label{Meq:XtoY}
\Map_\fX(I,F)\to\Map_\fY(I,F).
\end{equation}
\end{proposition}

We proceed as follows. Proposition~\ref{Mprp:XtoY} is proved in Subsection~\ref{ss:proof-XtoY}.
It implies that our main universality theorem 
\ref{Mthm:uni-prop} follows from Theorem~\ref{Mthm:uni-prop-0}.


The proof of Theorem~\ref{Mthm:uni-prop-0} is presented in Section \ref{sec:proof_univ_prop}.
\begin{lemma}
\label{Mlem:preexact}
Any faithful SM functor $F:\cD\to\cA$ to a tensor category is pre-exact.
\end{lemma}
\begin{proof}
Let $A\to B$ be a deflation in $\cD$.
\InnaB{By property \eqref{req_subcat_D3}}, there exists a
nonzero object $Z\in\InnaB{\cD}$ such that the map $A\otimes Z\to B\otimes Z$
 is split. This implies that $F(A)\otimes F(Z)\to F(B)\otimes F(Z)$ is surjective. The object $F(Z)$ is nonzero, therefore, faithfully flat. This implies that $F(A)\to F(B)$ is surjective.

The second part of the claim is proved similarly.
\end{proof}

\begin{lemma}\label{lem:pre_exact_to_exact}
 \InnaA{Let $U: \cV \to \cA$ be a pre-exact additive
\InnaA{$\C$-linear} SM functor, whose restriction to $\cD$ is faithful. Then 
$U$ is exact and faithful.} 
\end{lemma}
\begin{proof}
 \InnaA{We use the ``splitting of epimorphisms'' property \eqref{req_subcat_D3} 
of the 
categories $\cD, \cV$. Indeed, given a short exact sequence 
$$ 0 \to X \to Y \to Z \to 0$$ in $\cV$, we need to show 
that its image under $U$ is exact. The ``splitting of epimorphisms'' property 
implies that there exists an object $D \in \cD$ such that $Y\otimes D \to 
Z\otimes D \to 0$ splits. Since $U$ is $\C$-linear (hence preserves 
direct sums) and SM, the sequence
$$ 0 \to U(D) \otimes U(X) \to U(D) \otimes U(Y) \to U(D) \otimes U(Z) \to 0$$ 
is split exact. Since $U(D) \neq 0$, the object $U(D)$ is fully faithful in 
$\mathcal A$, and we conclude that the sequence $$ 0 \to  U(X) \to 
U(Y) \to  U(Z) \to 0$$ is exact as well. 

To show that $U$ is faithful, we recall that an exact functor $U$ is faithful iff for any object $L \in \cV$, $U(L) \neq0$. Let $L \in \cV$. Due to the presentation property \eqref{req_subcat_D2}, there exist $T, T' \in \cD$ and $f \in \Hom_{\cD}(T, T')$ such that $Im(f) = L$. Since $U$ is faithful on $\InnaB{\cD}$, $U(f) \neq 0$ and hence $U(L) = Im(U(f)) \neq 0$. }
\end{proof}

\subsection{Language of multicategories}

The best way to avoid taking care of various commutativity and associativity constrains, while working with SM categories, is to use the language of multicategories.

Let us remind some basic definitions.

\begin{definition}
A multicategory $\cC$ consists of the following data.
\begin{itemize} 
\item A collection of objects $\Ob\cC$.
\item A set $\Hom_\cC(\{x_i\},y)$ assigned to
any collection of objects $\{x_i\}_{i\in I}$ numbered by a finite set $I$ and to an object $y$.
\item Compositions
\begin{equation}\label{eq:multicomposition}
\Hom_\cC(\{y_j\},z)\times\prod_{j\in J}\Hom_\cC(\{x_i\}_{i\in f^{-1}(j)},y_j)\to\Hom_\cC(\{x_i\}_{i\in I},z),
\end{equation}
for any map $f:I\to J$ of finite sets.
\end{itemize}
The compositions should be associative and the sets $\Hom_\cC(\{x\},x)$
should have unit elements with the standard properties, see \cite[2.1.1]{L}.

\end{definition}

Any multicategory $\cC$ has an underlying category $\cC_1$ obtained by discarding all
non-unary operations. Any collection $\{x_i\}_{i\in I}$ defines a functor $\cC_1\to\Set$ carrying $y$ to $\Hom(\{x_i\},y)$. If this functor is (co)representable,
a representing object can be called $\bigotimes_{i\in I}x_i$. Finally, for each 
map $f:I\to J$ of finite sets a canonical map in $\cC_1$ is defined
\begin{equation}\label{eq:Mmulti}
\bigotimes_{j\in J}(\bigotimes_{i:f(i)=j}x_i)\to\bigotimes_{i\in I}x_i.
\end{equation}

\begin{definition}\label{dfn:sm-via-multi}
A multicategory $\cC$ is called SM category if all functors described above are 
corepresentable, and if all maps \eqref{eq:Mmulti} are isomorphisms.
\end{definition}

A functor $\cC\to\cD$ between multicategories is defined in an obvious way. 
Such a functor between SM categories is what is usually called a lax SM functor. 
It is a SM functor if a canonical morphism
\begin{equation}
\bigotimes_{i\in I}f(x_i)\to f(\bigotimes_{i\in I}x_i)
\end{equation} 
defined for any lax SM functor by universal property of tensor products,
is an isomorphism.

Finally, given two SM functors $f,g:\cC\to\cD$, a morphism $\theta:f\to g$
is just a collection of morphisms $\theta_x:f(x)\to g(x)$ for each $x\in\cC$
giving rise the the commutative diagrams
\begin{equation}
\begin{CD}
\Hom_{\cC}(\{x_i\},y) @>f>> \Hom_{\cD}(\{f(x_i\},f(y))  \\
@V{g}VV @VV{\theta(y)}V \\
\Hom_{\cD}(\{g(x_i)\},g(y))@>>{\theta(x_i)}>\Hom_{\cD}(\{f(x_i)\},g(y))
\end{CD}\notag
\end{equation}

for all $x_i,\ y\in\cC$.

\subsection{Proof of \texorpdfstring{Proposition~\ref{Mprp:XtoY}}{reduction statement}}
\label{ss:proof-XtoY}

We have to verify that for any $F:\cD\to\cA$ in $\fX$ the functor $\Map_\fX(I,F)\to\Map_\fY(I,F)$ induced by the forgetful functor $\#:\fX\to\fY$, is fully faithful
and essentially surjective.

{\sl Full faithfulness.} Let us verify that, given two arrows $(U,\theta)$ and $(U',\theta')$ in $\Map_\fX(I,F)$, any 2-arrow $\phi:(U,\theta)\to(U',\theta')$ in $\fY$ 
is automatically symmetric monoidal.

In other words, we have to verify that for any $M_i$ and $N$ in $\cV$   the diagram

\begin{equation}
\begin{CD}\label{Meq:2arrow}
\Hom_{\cV_t}(\{M_i\},N) @>U>> \Hom_{\cA_t}(\{UM_i\},UN)  \\
@V{U'}VV @VV{\phi(N)}V \\
\Hom_{\cA_t}(\{U'M_i\},U'N)@>>{\phi(M_i)}>\Hom_{\cA_t}(\{UM_i\},U'N)
\end{CD}\notag
\end{equation}
is commutative.
This is so for $M_i,\ N$ belonging to $I(\cD)$ as $\phi$ commutes with $\theta$ and $\theta'$. In order to verify the commutativity of the diagram (\ref{Meq:2arrow}) in general, choose epimorphisms $I(X_i)\twoheadrightarrow M_i$ and a monomorphism 
$N\hookrightarrow I(Y)$ with $X_i,\ Y$ in $\cD$. The diagram (\ref{Meq:2arrow}) will map injectively to the similar diagram for $X_i$ and $Y$ which is commutative.
This proves the claim.

{\sl Essential surjectivity.}
Now, given a morphism $(U,\theta):I\to F$ of functors, we have to extend it, up to isomorphism, to a morphism of SM functors. The functor $U:\cV\to \cA$ is given 
by a map $U:\Ob\ \cV\to\Ob\ \cA$ and a compatible collection of maps
\begin{equation}\label{eq:singlehoms}
\Hom_{\cV}(M,N)\to\Hom_\cA(UM,UN).
\end{equation}

We have to extend these data to a compatible collection of maps
\begin{equation}\label{eq:multihoms}
\Hom_{\cV}(\{M_i\},N) \to \Hom_{\cA}(\{UM_i\},UN).
\end{equation}

Choose for each finite collection of objects $\{X_i\}$ in $\cD$ a tensor product
$\otimes X_i$ with the universal element $\chi\in\Hom_{\cD}(\{X_i\},\otimes X_i)$. Similarly, we choose tensor product of each finite
 collection of objects in $\cV$ and in $\cA$.
 
Choose now for each $M_i\in\cV$ a presentation $I(X_i)\twoheadrightarrow M_i\hookrightarrow I(Y_i)$. We get (by universality of tensor products) the unique presentation
\begin{equation}\label{eq:pres-tensor}
\otimes I(X_i)\twoheadrightarrow \otimes M_i \hookrightarrow \otimes I(Y_i).
\end{equation}
We can now apply to \eqref{eq:pres-tensor} the functor $U$ and compare it
to the tensor product of presentations 
$UI(X_i)\twoheadrightarrow U(M_i) \hookrightarrow UI(Y_i)$. 
Taking into account that the functors $I$ and $F$ are symmetric monoidal, we get a canonical isomorphism $U(\otimes M_i)\to \otimes U(M_i)$.

We can now define the maps \eqref{eq:multihoms} as compositions
\begin{align}
\Hom_{\cV}(\{M_i\},N)=\Hom_{\cV}(\otimes M_i,N)\to\Hom_\cA(U(\otimes M_i),U(N))=\\ \nonumber=\Hom_\cA(\otimes U(M_i),U(N))=\Hom_\cA(\{UM_i\},UN).
\end{align}
Here we used the $=$ signs to denote canonical isomorphisms. 
Compatibility of  maps \eqref{eq:multihoms} with the compositions
\eqref{eq:multicomposition} directly follow (this is a long sequence 
of canonical morphisms) from the compatibility of  the maps
 \eqref{eq:singlehoms}  with the (usual) compositions. 
This means that $U$ extends to a SM functor\InnaA{, which we will denote $U$ 
as well. By Lemma \ref{lem:pre_exact_to_exact}, such a functor $U$ is exact.}

The isomorphism of functors $\theta:UI\to F$
is given by a collection of maps 
$\theta_X:UI(X)\to F(X)$
making the diagrams
\begin{equation}
\begin{CD}
\Hom_{\cD}(X,Y) @>UI>> \Hom_{\cA}(UI(X),UI(Y))  \\
@V{F}VV @VV{\theta_Y}V \\
\Hom_{\cA}(F(X),F(Y))@>>{\theta_X}>\Hom_{\cA}(UI(X),F(Y)) 
\end{CD}
\end{equation}
commutative.  The natural transformation $\theta$ is automatically symmetric monoidal since the maps \eqref{eq:multihoms} are expressed
via \eqref{eq:singlehoms}.
 \begin{remark}
  \InnaA{As it was proved in Lemma \ref{lem:pre_exact_to_exact}, extending a pre-exact faithful functor $U:\cD \to \cA$ we obtain a faithful exact functor $\cV \to \cA$. }
 \end{remark}



\section{Proof of \texorpdfstring{Theorem~\ref{Mthm:uni-prop-0}}{universal property}}\label{sec:proof_univ_prop}

We have a symmetric monoidal functor $I:\cD\to\cV$ satisfying the requirements 
\eqref{req_subcat_D1}--\eqref{req_subcat_D3}. In this 
section we will prove that the 
groupoid
$\Map_\fY(I,F)$ is contractible for any \InnaA{$\C$-linear faithful} pre-exact 
functor $F:\cD\to\cA$ into a
$\C$-linear abelian category. This means that a \InnaA{$\C$-linear} pre-exact 
functor 
$F:\cD\to\cA$ extends to a \InnaA{$\C$-linear} pre-exact functor $\cV\to\cA$ in 
an essentially unique
way. We will prove this in two steps. First of all, we will verify that
$\Map_\fY(I,F)$ is nonempty, that is that the functor $F:\cD\to\cA$ extends to
a pre-exact functor $U:\cV\to\cA$. Then we will prove that any two such extensions are connected by a unique isomorphism.

The idea of the construction of $U$ is very simple: we use existence of presentation of an object $X\in\cV$ as an image of $I(f)$, $f:P\to Q$ in $\cD$, to define $U(X)$ as an image of $F(f):F(P)\to F(Q)$.  One should be careful, however, keeping track of the choices involved. 

We will first describe
our bookkeeping device --- the collection of categories of presentations
for each arrow of $\cV$.

\subsection{Categories \texorpdfstring{$\cC_{f,\alpha,\beta,\alpha',\beta'}$}{of diagrams}}

\subsubsection{}
Assign to each map $f:M\to N$ in $\cV$, together with a choice of presentations 
$I(X)\stackrel{\alpha}{\twoheadrightarrow} M\stackrel{\beta}{\hookrightarrow} I(Y)$ 
and 
 $I(X')\stackrel{\alpha'}{\twoheadrightarrow} N\stackrel{\beta'}{\hookrightarrow} I(Y')$
 a category  $\cC_{f,\alpha,\beta,\alpha',\beta'}$ defined as follows.
\begin{itemize}
\item Its objects are
the diagrams 
\begin{equation}
\label{eq:obj-C}
 \xymatrix{&{} &I(X)  \ar@{->>}[r]^{\alpha} &M \ar[dd]^f \ar@{^{(}->}[r]^{\beta} &I(Y) \ar[rd] &{} \\ &I(R) \ar@{->>}[ur] \ar[dr] &{} &{} &{} &I(S) \\ &{} &I(X')  \ar@{->>}[r]_{\alpha'} &N \ar@{^{(}->}[r]_{\beta'} &I(Y') \ar@{^{(}->}[ru] &{} } 
\end{equation}
(pay attention which of the arrows are supposed to be injective and which are surjective!). Such diagram will be usually denoted for simplicity
as  $(R,S)$.
\item An arrow $(R_1,S_1)\to (R_2,S_2)$ is given by a pair of arrows
$R_1\to R_2$ and $S_2\to S_1$ in $\cD$ so that the diagram below is commutative.
\begin{equation}
\label{eq:ar-C}
 \xymatrix{&{} &{} &I(X)  \ar@{->>}[r]^{\alpha} &M \ar[dd]^f \ar@{^{(}->}[r]^{\beta} &I(Y) \ar[rd] \ar[rrd] &{} &{} \\ &I(R_1) \ar[r] \ar@{->>}[urr] \ar[drr] &I(R_2) \ar@{->>}[ur] \ar[dr] &{} &{} &{} &I(S_2) \ar[r] &I(S_1) \\ &{} &{} &I(X')  \ar@{->>}[r]_{\alpha'} &N \ar@{^{(}->}[r]_{\beta'} &I(Y') \ar@{^{(}->}[ru] \ar@{^{(}->}[rru] &{} &{} }.
\end{equation} 
Composition of the arrows is obvious.
\end{itemize} 

\subsubsection{} We will prove later that the categories 
$\cC_{f,\alpha,\beta,\alpha',\beta'}$ are nonempty and have connected nerve,
see Proposition~\ref{prop:C-connected}.

Let us now show how contractibility of the nerves can be used
in constructing the lifting of $F$.

\begin{itemize}
\item
Choose for each object $M\in\cV$ a presentation
$I(X)\stackrel{\alpha}{\twoheadrightarrow} M\stackrel{\beta}{\hookrightarrow} I(Y)$ and define $U(M)$ by a decomposition
$F(X)\stackrel{\alpha}{\twoheadrightarrow} U(M)\stackrel{\beta}{\hookrightarrow} F(Y)$. The object $U(M)$ so defined is defined uniquely up to unique isomorphism.

\item For any  map $f:M\to N$, define $U(f)$ by the diagram
\begin{equation}
\label{eq:Uobj-C}
 \xymatrix{&{} &F(X)  \ar@{->>}[r]^{\alpha} &U(M) \ar[dd]^{U(f)} \ar@{^{(}->}[r]^{\beta} &F(Y) \ar[rd] &{} \\ &F(R) \ar@{->>}[ur] \ar[dr] &{} &{} &{} &F(S) \\ &{} &F(X')  \ar@{->>}[r]_{\alpha'} &U(N) \ar@{^{(}->}[r]_{\beta'} &F(Y') \ar@{^{(}->}[ru] &{} } .
\end{equation}
This is possible as the category
$\cC_{f,\alpha,\beta,\alpha',\beta'}$ is nonempty. The result is independent of the choice of an object 
$(R,S)\in\cC_{f,\alpha,\beta,\alpha',\beta'}$ as the nerve of the category is 
connected and any arrow \eqref{eq:ar-C} in it induces the same map $U(M)\to 
U(N)$.
\end{itemize}

It remains to verify a number of properties of the construction.
It is done in the following subsection.

\subsection{End of the proof}

\subsubsection{\texorpdfstring{$U$}{U} is a functor} We  have to verify that 
$U(g)\circ U(f)=U(g\circ f)$.

Let $(R_1, S_1)$ be an object in the category 
$\mathcal{C}_{f, \alpha, \beta, \alpha', \beta'}$ and let $(R_2, S_2)$ be an
 object in the category $\mathcal{C}_{f, \alpha', \beta', \alpha'', \beta''}$,
see \eqref{eq:composition-C}. 
We then have triples $$ R_1 \rightarrow X' \twoheadleftarrow R_2, \; \; S_2
 \leftarrow Y' \hookrightarrow S_1.$$

Consider the fiber product $I(R_1) \times_{I(X')} I(R_2)$ and the cofiber coproduct \newline
 $I(S_1)\sqcup^{I(Y')} I(S_2)$; the existence of  presentations implies that we can choose objects $R, S$ in $\cD$ with morphisms 
$$ I(R) \twoheadrightarrow I(R_1) \times_{I(X')} I(R_2), \;\; I(S_1) \sqcup^{I(Y')} I(S_2) \hookrightarrow I(S).$$

Since epimorphisms  in $\cA$ are preserved by base change, and monomorphisms
are preserved by cobase change, we obtain the following commutative diagram.
\begin{equation}\label{eq:composition-C}
\xymatrix
{&{} &{} &I(X)  \ar@{->>}[r]^{\alpha} &M \ar[dd]^f \ar@{^{(}->}[r]^{\beta} &I(Y) \ar[rd] &{} &{} \\
 &{} &I(R_1) \ar@{->>}[ur] \ar[dr] &{} &{} &{} &I(S_1) \ar[rd] &{} \\
 &I(R) \ar@{->>}[ru] \ar[rd] &{} &I(X')  \ar@{->>}[r]_{\alpha'} &N \ar[dd]^{g} \ar@{^{(}->}[r]_{\beta'} &I(Y') \ar@{^{(}->}[ru] \ar[rd] &{} &I(S)  \\\ &{} &I(R_2) \ar@{->>}[ur] \ar[dr] &{} &{} &{} &I(S_2) \ar@{^{(}->}[ur] &{} \\ &{} &{} &I(X'')  \ar@{->>}[r]^{\alpha''} &L \ar@{^{(}->}[r]^{\beta''} &I(Y'') \ar@{^{(}->}[ru] &{} &{} }
\end{equation}

This commutative diagram implies that $(R, S)$ is an object in the category $\mathcal{C}_{f, \alpha, \beta, \alpha'', \beta''}$.

Diagram \eqref{eq:composition-C} gives rise to a diagram
\begin{equation}\label{eq:Ucomposition-C}
\xymatrix
{&{} &{} &F(X)  \ar@{->>}[r]^{\alpha} &U(M) \ar[dd]^{U(f)} \ar@{^{(}->}[r]^{\beta} &F(Y) \ar[rd] &{} &{} \\
 &{} &F(R_1) \ar@{->>}[ur] \ar[dr] &{} &{} &{} &F(S_1) \ar[rd] &{} \\
 &F(R) \ar@{->>}[ru] \ar[rd] &{} &F(X')  \ar@{->>}[r]_{\alpha'} &U(N) \ar[dd]^{U(g)} \ar@{^{(}->}[r]_{\beta'} &F(Y') \ar@{^{(}->}[ru] \ar[rd] &{} &F(S)  \\\ &{} &F(R_2) \ar@{->>}[ur] \ar[dr] &{} &{} &{} &F(S_2) \ar@{^{(}->}[ur] &{} \\ &{} &{} &F(X'')  \ar@{->>}[r]^{\alpha''} &U(L) \ar@{^{(}->}[r]^{\beta''} &F(Y'') \ar@{^{(}->}[ru] &{} &{} },
\end{equation}
where $U(f)$ and $U(g)$ are uniquely determined by the condition that they make the diagram commutative. This implies that their composition coincides with $U(gf)$.

\subsubsection{Independence of the choice of presentations}

The functor $U$ constructed above used a choice of presentation 
 $I(X)\stackrel{\alpha}{\twoheadrightarrow}M\stackrel{\beta}{\hookrightarrow}
I(Y)$ and of  decomposition
$F(X)\stackrel{\alpha}{\twoheadrightarrow}U(M)\stackrel{\beta}{\hookrightarrow}
F(Y)$ for each object $M\in\cV$.

Any (other) presentation 
\begin{equation}\label{eq:pres1}
I(X')\stackrel{\alpha'}{\twoheadrightarrow} M\stackrel{\beta'}{\hookrightarrow}
 I(Y')
\end{equation}
 of $M$, and, accordingly, a decomposition 
$ F(X')\stackrel{\alpha'}{\twoheadrightarrow} U'(M)\stackrel{\beta'}
{\hookrightarrow} F(Y'),$
gives rise to an isomorphism $U(M)\to U'(M)$, once more, since the category
$\cC_{\id,\alpha,\beta,\alpha',\beta'}$ has connected nerve.

This allows us to claim that any presentation \eqref{eq:pres1} gives rise
to a unique decomposition 
$$ F(X')\stackrel{\alpha'}{\twoheadrightarrow} U(M)\stackrel{\beta'}
{\hookrightarrow} F(Y').$$

That is, our construction does not depend, up to unique isomorphism, on the choices.

\subsubsection{Isomorphisms \texorpdfstring{$U\circ I(X)\to F(X)$}{}}
Choosing a trivial presentation $I(X)\to I(X)\to I(X)$ for $I(X)$,  we get 
a canonical isomorphism $U(I(X))\to F(X)$.

\subsubsection{Pre-exactness}
We claim that the functor $U$ preserves injective and surjective morphisms.
In fact, if, for instance, $f$ is surjective, one can choose 
$\alpha'=f\circ\alpha$, so that we can choose $R=X'=X$ in the notation 
of \eqref{eq:Uobj-C}. This implies $U(f)$ is surjective,. The dual statment us
similar.

\

We still have to prove  that the nerves of the categories
$\cC_{f,\alpha,\beta,\alpha',\beta'}$ are connected.

\begin{proposition} 
\label{prop:C-connected}
The category $\cC_{f,\alpha,\beta,\alpha',\beta'}$
is nonempty and has a connected nerve.
\end{proposition}
\begin{proof}
We  will first of all present an explicit construction of an object in
$\cC_{f,\alpha,\beta,\alpha',\beta'}$. Then we will prove that any other object
of this category is connected to it by a zigzag of arrows.

{\sl Existence of an object.} 

Given presentations
$$\begin{CD}I(X)@>\alpha>> M @>\beta>> I(Y)\end{CD}$$
and
$$\begin{CD}I(X')@>\alpha'>> N @>\beta'>> I(Y')\end{CD},$$
choose an epimorphism 
$I(R)\twoheadrightarrow I(X)\times_NI(X')$ and 
a monomorphism $\begin{CD}I(Y)\sqcup^MI(Y') @>\beta>> I(S)\end{CD}.$

It is easy to see that the resulting commutative diagram
  
$$ \xymatrix{&{} &X  \ar@{->>}[r]^{\alpha} &M \ar[dd]^f \ar@{^{(}->}[r]^{\beta} &Y \ar[rd] &{} \\ &R \ar@{->>}[ur] \ar[dr] &{} &{} &{} &S \\ &{} &X'  \ar@{->>}[r]_{\alpha'} &N \ar@{^{(}->}[r]_{\beta'} &Y' \ar@{^{(}->}[ru] &{} } $$ 
belongs to $\cC_{f,\alpha,\beta,\alpha',\beta'}$.

{\sl Connectedness.}

  Consider the object $(R, S)$ in $\mathcal{C}_{f, \alpha, \beta, \alpha', \beta'}$ we have just constructed.

  Let $(R', S')$ be any object in $\mathcal{C}_{f, \alpha, \beta, \alpha', \beta'}$. We will now prove that there exists an object $(R'', S'')$ in $\mathcal{C}_{f, \alpha, \beta, \alpha', \beta'}$ together with arrows $(R'', S'') \rightarrow (R', S')$, $(R'', S'') \rightarrow 
  (R, S)$.
  
By definition of objects in $\mathcal{C}_{f, \alpha, \beta, \alpha', \beta'}$, there are canonical morphisms $$R \longrightarrow X \times_N X' \longleftarrow R', \; \; S' \longleftarrow Y \sqcup_M Y' \longrightarrow S$$ Moreover, the construction of the object $(R, S)$ tells us that $R \rightarrow X \times_N X'$ is an epimorphism and $Y \sqcup_M Y' \longrightarrow S$ is a monomorphism.

We then consider the objects $$ R \times_{X \times_N X'} R' , \; \; S' \sqcup_{Y \sqcup_M Y'} S $$ in $\mathcal{A}$. 
Since fiber products preserve epimorphisms in abelian categories, 
the canonical morphism $$R \times_{X \times_N X'} R' \longrightarrow R'$$ is an epimorphism. For the same (dual) reason the canonical morphism 
$$ S \longrightarrow S' \sqcup_{Y \sqcup_M Y'} S $$ is a monomorphism.

Now fix two objects $R'', S''$ in $\cV$ with arrows 
$$R'' \twoheadrightarrow R \times_{X \times_N X'} R', \; \;  S' \sqcup_{Y \sqcup_M Y'} S \hookrightarrow  S''$$

Thus we obtain the following commutative diagram:
$$ \xymatrix{&{} &R \ar@{->>}[r] \ar[rdd] &X  \ar@{->>}[r]^{\alpha} &M \ar[dd]^f \ar@{^{(}->}[r]^{\beta} &Y \ar[r] \ar[rdd] &S \ar[rd] &{} \\ &R'' \ar[ur] \ar@{->>}[dr] &{} &{} &{} &{} &{} &S'' \\ &{} &R' \ar@{->>}[ruu] \ar[r] &X'  \ar@{->>}[r]_{\alpha'} &N \ar@{^{(}->}[r]_{\beta'} &Y' \ar@{^{(}->}[r] \ar@{^{(}->}[ruu] &S' \ar@{^{(}->}[ru] &{} } $$

It remains to show that $(R'', S'')$ is indeed an object in $\mathcal{C}_{f, \alpha, \beta, \alpha', \beta'}$, i.e. that $R'' \rightarrow X$ is an epimorphism and that $Y' \rightarrow S''$ is a monomorphism.

The construction of $(R'', S'')$ described above implies that we have the following commutative squares:
 $$\xymatrix{ &R'' \ar[r] \ar@{->>}[d] &R \ar@{->>}[d] &{} &Y' \ar@{^{(}->}[r] \ar@{^{(}->}[d] &S \ar[d] \\ &R' \ar@{->>}[r] &X &{} &S' \ar@{^{(}->}[r] &S''}$$
so the above statement clearly holds. This completes the proof.
\end{proof}

\section{Deligne's conjecture}\label{sec:Deligne_conj}
%
%
%
%
%

\subsection{Introduction}
In this section we show that universality of tensor category $\cV_t$
in the sense of Theorem~\ref{Mthm:uni-prop} easily implies positive answer
to Deligne's question \cite[Question (10.18)]{D}.

Let $t \in \C$. 
Let $\mathcal{T}$ be \InnaB{a} tensor category, and let $X$ be an object in $\mathcal{T}$
 of dimension $t$. Consider the category $\rp_{\mathcal{T}}(GL(X), \varepsilon)$
 of $\pi(\mathcal{T})$-equivariant representations of the affine group scheme
 $GL(X)$ in $\mathcal{T}$, see \cite[Section 10.8]{D} or below. This category is a tensor category containing $X$.
 Since $X$ has dimension $t$, it gives rise to a SM functor 
$$F_X:
\cD_t\longrightarrow \rp_{\mathcal{T}}(GL(X), \varepsilon) \;\;\; X_t \mapsto X
$$


\begin{lemma}
 The functor $F_X$ is faithful if and only if the object $X\in\cT$ is not annihilated by any Schur functor.
\end{lemma}
\begin{proof}
Schur functors are given by idempotents $e_{\lambda} \in \End(X^{\otimes |\lambda|})$ which are images under $F_X$ of the corresponding idempotents $\varepsilon_{\lambda} \in \End(X_t^{\otimes  |\lambda|})$. Thus $$F_X(\varepsilon_{\lambda}) \neq 0 \; \; \Leftrightarrow \;\;S^{\lambda} X \neq 0$$ for any $\lambda$. We need to show that
$$\forall \lambda, F_X(\varepsilon_{\lambda}) \neq 0 \; \; \Leftrightarrow \;\; F_X \text{ is faithful}$$

Indeed, recall that we have canonical isomorphisms 
$$\Hom_{\cD_t}(X_t^{\otimes r} \otimes X_t^{* \otimes s} , X_t^{\otimes r'} \otimes X_t^{* \otimes s'} ) \cong  \Hom_{\cD_t}(X_t^{\otimes r+s'} , X_t^{\otimes r'+s})$$ The latter space is generated by idempotents $\{\varepsilon_{\lambda}\}_{|\lambda|=r+s'}$ if $r+s' = r'+s$, and is zero otherwise (cf. \cite[Section 10]{D}). 
\end{proof}

\begin{theorem}\label{thrm:Deligne_conj}

\mbox{}

\begin{enumerate}[label=(\alph*)]
 \item  If $X$ is not annihilated by any Schur functor then $F_X$ uniquely 
factors through the embedding $I:\cD_t\to\cV_t$ and gives rise to an equivalence of tensor categories
 $$\cV_t\longrightarrow \rp_{\mathcal{T}}(GL(X), \varepsilon)$$
 sending $V_t$ to $X$.
 \item If $X$ is annihilated by some Schur functor then there exists a unique pair  $m, n \in \mathbb{Z}_+$, $m-n=t$, such that $F_X$ factors through the SM functor $\cD_t\longrightarrow \rp(\gl(m|n))$ and gives rise to an equivalence of tensor categories
 $$ \rp(\gl(m|n)) \longrightarrow \rp_{\mathcal{T}}(GL(X), \varepsilon)$$
 sending the standard representation $\mathbb{C}^{m |n}$ to $X$.
\end{enumerate}

\end{theorem}
A proof will be given in Subsection \ref{ssec:proof_Del_conj}.
\begin{remark}
 Note that for $t \notin \Z$, this theorem was proved by V. Ostrik in
 \cite[Appendix B]{D}, in a similar manner. Our proof relies only on the 
universal property of $\cV_t$, and therefore works for any $t \in \C$.
\end{remark}

\subsection{Algebraic groups in tensor categories and their representations:
reminder}

The content of this subsection is mostly taken from \cite{D1}, Sect. 7.

\subsubsection{}
Let $\cT$ be a tensor category. The category $\Ind\cT$ inherits a symmetric monoidal structure. Algebraic groups in $\cT$ (or $\cT$-algebraic groups) are group objects in the category of $\cT$-affine schemes; thus, these are just commutative Hopf algebra objects in $\Ind\cT$. Yoneda lemma allows one to
identify $\cT$-algebraic groups with the corresponding "functors of points" ---
these are corepresentable functors $\Com(\Ind\cT)\to\Grps$ from commutative algebras in $\Ind\cT$ to groups. A representation $V$ of a $\cT$-algebraic group 
is an object $V\in\cT$ endowed with a structure of left comodule of the appropriate Hopf algebra.

Given a tensor functor $F:\cT\to\cT'$ and a $\cT$-algebraic group $G$, the image
$F(G)$ is obtained by applying the functor $F$ to the corresponding Hopf algebra
object of $\Ind\cT$.

In case the tensor functor $F:\cT\to\cT'$ is right exact, 
the $\cT'$-algebraic group $F(G)$ can be also described in terms of 
the functor of points. Recall that a right exact tensor functor 
$F:\cT\to\cT'$ induces a tensor functor $F:\Ind\cT\to\Ind\cT'$ 
commuting with small colimits. 

By the Adjoint Functor Theorem (see \cite{F}),
$F$ admits a right adjoint functor $F^!:\Ind\cT'\to\Ind\cT$ which is
automatically lax symmetric monoidal 
(see Definition~\ref{dfn:sm-via-multi} and a discussion following it).

Then the $\cT'$-algebraic group $F(G)$ defines a functor
$\Com(\Ind\cT')\to\Grps$ given by the formula
\begin{equation}\label{eq:FofG}
F(G)(B)=G(F^!(B)).
\end{equation}

\subsubsection{}
For an algebra $A\in\Com(\Ind\cT)$ the category of $A$-modules $\Mod_A$ is 
defined in a standard way. The functor $i_A:\cT\to\Mod_A$ carries $X\in\cT$ to $A\otimes X$. 

The fundamental group $\pi(\cT)$ is defined as the $\cT$-algebraic group
defined by the functor of points by the formula
\begin{equation}
A\mapsto \Aut^{\otimes}(i_A:\cT\to\Mod_A).
\end{equation} 

The fundamental group of a tensor category is affine (that is, the functor of points is corepresentable) if, for instance, the base field is perfect and the category is pre-Tannakian (see \cite[Section 8.1]{D1} for definition).

\subsubsection{}
Let $X\in\cT$. Define a functor $\Com(\Ind\cT)\to\Grps$ by the formula
  
$$ A\mapsto\Aut(i_A(X)).$$  

Again, if the base field is perfect and the category is pre-Tannakian, then this functor in corepresentable 
by a $\cT$-algebraic group denoted $GL(X)$. One has an obvious
evaluation map $\epsilon:\pi(\cT)\to GL(X)$ given, on the level of functors of
 points, by the assignment of $\theta(X):i_A(X)\to i_A(X)$ to an
 automorphism $\theta$ of the functor $i_A:\cT\to\Mod_A$.

In particular, the homomorphism $\epsilon$ described above endows any object $X\in\cT$ with a canonical action of $\pi(\cT)$. 

For example, $\pi(\sVect)$ is the group of two elements, with the nontrivial element acting on any super space $V$ by $1$ on its even part and by $-1$ on the odd part.

\subsubsection{Functoriality}
\label{sss:functoriality}
 Given an exact SM functor $F:\cT\to\cT'$
and an object $X\in\cT$, one has a natural homomorphism
\begin{equation}\label{eq:functoriality-gl}
F(GL(X))\to GL(F(X))
\end{equation}
defined by the functors of points as follows. For a fixed
$B\in\Com(\Ind\cT')$ the group $GL(F(X))(B)$ is the automorphism group 
of the free $B$-module $B\otimes F(X)$. The group $F(GL(X))(B)$
is the automorphism group of the free $F^!(B)$-module 
$F^!(B)\otimes X$.
Given $\alpha\in\Aut_{F^!(B)}(F^!(B)\otimes X)$, we can apply $F$ 
to get $F(\alpha)\in\Aut_{FF^!(B)}(FF^!(B)\otimes F(X))$. Making base change along the ring homomorphism $FF^!(B)\to B$, we get an element of $GL(F(X))(B)$.

The homomorphism $F(GL(Z))\to GL(F(Z))$ is an isomorphism
--- this follows from the explicit construction of the Hopf algebras in $\mathcal{T}$ and $\mathcal{T}'$ corresponding to the affine group scheme $GL(Z)$ and $GL(F(Z))$ (see for example \cite{Et}). 

The functor $F:\cT\to\cT'$ induces a map of the respective fundamental groups
\begin{equation}\label{eq:functoriality-pi}
\pi(\cT')\to F(\pi(\cT))
\end{equation}
defined as follows. The algebraic group $F(\pi(\cT))$ defines,
 according to \eqref{eq:FofG}, the functor of points carrying
 $B\in\Com(\Ind\cT')$ to the group $Aut^\otimes(\phi_B)$ where the
 functor $\phi_B:\cT\to\Mod_{F^!(B)}$ is defined by the formula
 $\phi_B(Y)=F^!(B)\otimes Y$. Deligne suggests another functor 
of points
$$ B\mapsto Aut^\otimes(\psi_B),$$
with $\psi_B:\cT\to\Mod_B$ defined by the formula 
$\psi_B(X)=B\otimes F(X)$. One has a canonical morphism
$Aut^\otimes(\phi)\to Aut^\otimes(\psi)$ defined in the same way as 
\eqref{eq:functoriality-gl}. It is proven in \cite[8.6]{D1}, that this map is also an isomorphism.

Taking into account the alternative description of the functor of points for $F(\pi(\cT))$, one defines the map \eqref{eq:functoriality-pi} in terms of functors of points as follows. Denote 
$\phi'_B:\cT'\to\Mod_B$ the functor carrying $Y\in\cT'$ to $B\otimes Y$, so that $\pi(\cT')$ is defined by the tensor automorphisms of $\phi'$. 
Given $B\in\Com(\cT')$ and $\alpha\in Aut^\otimes(\phi'_B)$, 
we construct its image in $Aut^\otimes(\psi)$ by composing $\alpha$ with $F:\cT\to\cT'$.

\subsubsection{}
Given a $\cT$-algebraic group $G$
and a homomorphism $\epsilon:\pi(\cT)\to G$, one defines the tensor
category $\rp_\cT(G,\epsilon)$ as the full subcategory of representations
of $G$ such that the representation of $\pi(\cT)$ obtained via the pullback along $\epsilon$, is the standard one.

In case $\cT=\sVect$ and $G=GL(m|n)$, the category $\rp_\cT(G,\epsilon)$ is precisely
the "small" category $\rp(\gl(m|n))$ defined in Section \ref{ssec:rep_superalgebras}.

\mbox{}

The following result of Deligne  is a relative version of the Tannakian reconstruction (Deligne, \cite[Theorem 8.17]{D1}):
\begin{proposition}\label{prop:Delignes_fund_grps}
 Let $\mathcal{T}, \mathcal{T}'$ be two pre-Tannakian tensor categories and let $F: \mathcal{T} \rightarrow \mathcal{T}'$ be an exact SM functor. 
This functor induces an equivalence
 $\mathcal{T}\rightarrow\rp_{\mathcal{T}'}(F(\pi(\mathcal{T})), \epsilon)$. 
\end{proposition}

\subsection{Proof of the Deligne conjecture}\label{ssec:proof_Del_conj}

The following result is an easy consequence of Proposition~\ref{prop:Delignes_fund_grps}.
 
\begin{corollary}\label{lem:equiv_relative_repr_cat}
  Let $\mathcal{T}, \mathcal{T}'$ be two pre-Tannakian tensor categories and let $Z$ be a generating object of $\mathcal{T}$ in the sense that 
the canonical map $\pi(\mathcal{T})\to GL(Z)$ is an isomorphism. 
Let $F: \mathcal{T} \rightarrow \mathcal{T}'$ be an exact SM functor. 
This functor induces an equivalence of tensor categories 
$$F: \mathcal{T} \longrightarrow \rp_{\cT'}(GL(F(Z)), \epsilon).$$ 
\end{corollary}

\begin{proof}[Proof of Corollary \ref{lem:equiv_relative_repr_cat}:]\label{prooflem:equiv_relative_repr_cat}

By Proposition \ref{prop:Delignes_fund_grps}, we only need to prove that there is a $\pi(\mathcal{T}')$-equivariant isomorphism $F(\pi(\mathcal{T})) \cong GL(F(Z))$.


Since $\pi(\cT)\cong GL(Z)$, it remains to verify that the following diagram of affine groups schemes  over $\cT'$ is commutative.

\begin{equation}\label{eq:functoriality-pi-GL}
\xymatrix{  &F(\pi(\mathcal{T})) \ar[r] &F(GL(Z))\ar[d] \\
  &\pi(\mathcal{T}') \ar[u] \ar[r]& GL(F(Z))
  }
\end{equation}

Let us write down the respective functors $\Com(\cT')\to \Grps$. 

We get the diagram where the functors $\phi_B,\ \phi'_B$ and $\psi_B$ 
have the same meaning as in \ref{sss:functoriality}. 

\begin{equation}\label{eq:functoriality-pi-GL-pts} 
\xymatrix{  
&\left(B \mapsto Aut^{\otimes}(\phi_B)\right)\ar[d]^\cong\ar[r]^-{\epsilon_Z}
&\left(B\mapsto Aut_{F^!(B)}(F^!(B)\otimes Z)\right)\ar[dd]\\
&\left(B\mapsto Aut^\otimes(\psi_B)\right)\ar[dr]^{\epsilon'_Z}\\
&\left(B\mapsto Aut^\otimes(\phi'_B)\right)\ar[u]\ar[r]^-{\epsilon_{F(Z)}}
&\left(B\mapsto Aut_B(B\otimes F(Z))\right) 
}
\end{equation}
and the maps $\epsilon_Z,\ \epsilon'_Z$ and $\epsilon_{F(Z)}$ are defined as evaluations
at $Z,\ Z$, and $F(Z)$ respectively.

One easily verifies that the diagonal arrow $\epsilon'_Z$ cuts the diagram into a commutative triangle and a commutative square. This implies the commutativity of the whole diagram.


\end{proof}

\subsubsection{}
\begin{proof}[Proof of Theorem \ref{thrm:Deligne_conj}:]

 \begin{enumerate}[label=(\alph*)]
 \item Assume $X$ is not annihilated by any Schur functor. By 
Theorem~\ref{Mthm:uni-prop}, the functor $F_X$ extends uniquely (up to unique isomorphism) to an exact SM functor $$U_X: \cV_t \longrightarrow \rp_\cT(GL(X),\epsilon). $$
 The statement of Theorem \ref{thrm:Deligne_conj} will follow from Proposition \ref{lem:equiv_relative_repr_cat} once we show that the canonical homomorphism of affine group schemes in $\cV_t$
 $$\pi(\cV_t) \rightarrow GL(V_t)$$
 is an isomorphism.
 
Let us compare the respective functors of points $\Com(\cV_t)\longrightarrow\Grps$.

The functor represented by $\pi(\cV_t)$ 
carries $A\in\Com(\cV_t)$ to the automorphism group of the tensor
functor $i_A:\cV_t\to\Mod_A$.

The functor represented by $GL(V_t)$ carries $A$ to the automorphism group
of $A\otimes V_t$ considered as $A$-module. 

The map $\pi(\cV_t)\to GL(V_t)$ is defined by evaluation of any automorphism of $i_A$ at $V_t\in\cV_t$. This is an isomorphism by universality of $\cV_t$,
see~Theorem~\ref{Mthm:uni-2}: we have an equivalence

$$\Fun^\ex(\cV_t,\Mod(A))\longrightarrow\Fun^\faith(\cD_t,\Mod(A))$$

which yields an isomorphism
$$ \Aut^{\InnaB{\otimes}}(i_A)\longrightarrow\Aut_A(A\otimes V_t).$$
This completes the proof of Part (a).

\item Assume $X$ is annihilated by some Schur functor $S^{\lambda}$.
By \cite[Proposition 0.5(ii)]{D2}, this means that any subquotient of a finite direct 
sum of mixed tensor powers of $X$ is annihilated by some Schur functor.


Recall that the objects of the category $\rp_{\mathcal{T}}(GL(X), \epsilon)$ are
subquotients of direct sums of mixed tensor powers of $X$. 
Therefore, the category $\rp_{\mathcal{T}}(GL(X), \epsilon)$ satisfies the
 conditions of \cite[Theorem 0.6]{D2} and thus is super-Tannakian, i.e. possesses
 a super-fiber functor 
$$S:\rp_{\mathcal{T}}(GL(X), \epsilon) \longrightarrow \sVect.$$

The image of the object $X$ under the super-fiber functor is then isomorphic to the super vector space $\mathbb{C}^{m|n}$ for some $m, n \in \mathbb{Z}$ such that $m-n = t$.

Applying Proposition \ref{lem:equiv_relative_repr_cat} to the super-fiber functor $S$, we obtain an equivalence of categories 
$$\rp_{\mathcal{T}}(GL(X), \epsilon) \stackrel{\sim}{\longrightarrow}
\rp_{\sVect}(GL(m|n), \epsilon)=\rp(\gl(m|n)). $$

\end{enumerate}
This completes the proof of Theorem \ref{thrm:Deligne_conj}.
\end{proof}

We can apply Theorem~\ref{thrm:Deligne_conj} to the following 
construction, due to Deligne (see \cite[Section 10]{D}).

Choose $t_1 \in \C-\Z$. Consider the tensor category 
$\mathcal{T} :=\cD_{t_1} \otimes \cD_{t-t_1}$ and the object 
$X := V_{t_1}\otimes\triv\oplus\triv\otimes V_{t-t_1}$.

%

The object $X$  has dimension $t$, so it gives rise to a tensor functor
$\cD_t\to\cT$. Deligne proves \cite[Conjecture (10.17)]{D1} that this functor is fully faithful. By Theorem~\ref{thrm:Deligne_conj}, we obtain:

\begin{corollary}\label{cor:Deligne_conj1}
 For any $t_1 \in \C-\Z$, there is a unique canonical equivalence
 $$\cV_t\longrightarrow\rp_{\cD_{t_1}\otimes\cD_{t-t_1}}(GL(X), \epsilon)$$
 carrying $V_t$ to $X$.
\end{corollary}


\section*{Acknowledgements}
We would like to thank P. Etingof and V. Ostrik for stimulating discussions and 
advice, and J. Bernstein for useful discussions. We are also thankful to the
referee, to K. Coulembier \InnaD{and to A. Savage} for helpful discussions and for pointing out numerous typos in the previous versions of the 
manuscript. The first author was supported by the ERC under grant agreement No 669655 (PI: Prof. David Kazhdan). The second author was supported by the grant ISF-446/15. The third
author was supported by NSF grant 1701532.

\end{document}